\documentclass[11pt]{amsart}

\usepackage{amsfonts}
\usepackage{amsmath} 
\usepackage{latexsym}
\usepackage{amssymb}
\usepackage{verbatim}
\usepackage[usenames]{color}
\usepackage{hyperref}
\usepackage{url}
\usepackage{mathrsfs}
\usepackage{tikz,tikz-qtree,ifthen}

\DeclareMathOperator{\UP}{UP}
\DeclareMathOperator{\lh}{lh}

\begin{document}

\newcommand{\s}{\sigma}
\newcommand{\al}{\alpha}
\newcommand{\om}{\omega}
\newcommand{\be}{\beta}
\newcommand{\la}{\lambda}
\newcommand{\vp}{\varphi}

\newcommand{\bo}{\mathbf{0}}
\newcommand{\bone}{\mathbf{1}}

\newcommand{\sse}{\subseteq}
\newcommand{\contains}{\supseteq}
\newcommand{\forces}{\Vdash}

\newcommand{\FIN}{\mathrm{FIN}}
\newcommand{\Fin}{\mathrm{Fin}}
\newcommand{\fin}{\mathrm{fin}}

\newcommand{\fa}{\mathfrak{a}}

\newcommand{\ve}{\vee}
\newcommand{\w}{\wedge}
\newcommand{\bv}{\bigvee}
\newcommand{\bw}{\bigwedge}
\newcommand{\bcup}{\bigcup}
\newcommand{\bcap}{\bigcap}

\newcommand{\tre}{\trianglerighteq}
\newcommand{\tr}{\vartriangleright}
\newcommand{\tle}{\trianglelefteq}
\newcommand{\tl}{\vartriangleleft}

\newcommand{\rgl}{\rangle}
\newcommand{\lgl}{\langle}
\newcommand{\lr}{\langle\ \rangle}
\newcommand{\re}{\restriction}

\newcommand{\bB}{\mathbb{B}}
\newcommand{\bP}{\mathbb{P}}
\newcommand{\bR}{\mathbb{R}}
\newcommand{\bW}{\mathbb{W}}
\newcommand{\bX}{\mathbb{X}}
\newcommand{\bN}{\mathbb{N}}
\newcommand{\bQ}{\mathbb{Q}}
\newcommand{\bS}{\mathbb{S}}
\newcommand{\St}{\tilde{S}}
   
\newcommand{\sd}{\triangle}
\newcommand{\cl}{\prec}
\newcommand{\cle}{\preccurlyeq}
\newcommand{\cg}{\succ}
\newcommand{\cge}{\succcurlyeq}
\newcommand{\dom}{\mathrm{dom}}
\newcommand{\ran}{\mathrm{ran}}

\newcommand{\lra}{\leftrightarrow}
\newcommand{\ra}{\rightarrow}
\newcommand{\llra}{\longleftrightarrow}
\newcommand{\Lla}{\Longleftarrow}
\newcommand{\Lra}{\Longrightarrow}
\newcommand{\Llra}{\Longleftrightarrow}
\newcommand{\rla}{\leftrightarrow}
\newcommand{\lora}{\longrightarrow}
\newcommand{\E}{\mathrm{E}}
\newcommand{\rank}{\mathrm{rank}}
\newcommand{\lefin}{\le_{\mathrm{fin}}}
\newcommand{\Ext}{\mathrm{Ext}}
\newcommand{\lelex}{\le_{\mathrm{lex}}}
\newcommand{\depth}{\mathrm{depth}}

\newcommand{\Erdos}{Erd{\H{o}}s}
\newcommand{\Pudlak}{Pudl{\'{a}}k}
\newcommand{\Rodl}{R{\"{o}}dl}
\newcommand{\Proml}{Pr{\"{o}}ml}
\newcommand{\Fraisse}{Fra{\"{i}}ss{\'{e}}}
\newcommand{\Sokic}{Soki{\'{c}}}
\newcommand{\Nesetril}{Ne{\v{s}}et{\v{r}}il}

\newtheorem{thm}{Theorem}  
\newtheorem{prop}[thm]{Proposition} 
\newtheorem{lem}[thm]{Lemma} 
\newtheorem{cor}[thm]{Corollary} 
\newtheorem{fact}[thm]{Fact}    
\newtheorem{facts}[thm]{Facts}      
\newtheorem*{thmMT}{Main Theorem}
\newtheorem*{thmMTUT}{Main Theorem for $\vec{\mathcal{U}}$-trees}
\newtheorem*{thmnonumber}{Theorem}
\newtheorem*{mainclaim}{Main Claim}
\newtheorem{claim}{Claim}
\newtheorem*{claim1}{Claim $1$}
\newtheorem*{claim2}{Claim $2$}
\newtheorem*{claim3}{Claim $3$}
\newtheorem*{claim4}{Claim $4$}

\theoremstyle{definition}   
\newtheorem{defn}[thm]{Definition} 
\newtheorem{example}[thm]{Example} 
\newtheorem{conj}[thm]{Conjecture} 
\newtheorem{prob}[thm]{Problem} 
\newtheorem{examples}[thm]{Examples}
\newtheorem{question}[thm]{Question}
\newtheorem{problem}[thm]{Problem}
\newtheorem{openproblems}[thm]{Open Problems}
\newtheorem{openproblem}[thm]{Open Problem}
\newtheorem{conjecture}[thm]{Conjecture}
\newtheorem*{problem1}{Problem 1}
\newtheorem*{problem2}{Problem 2}
\newtheorem*{problem3}{Problem 3}
\newtheorem*{notation}{Notation}

\theoremstyle{remark} 
\newtheorem*{rem}{Remark} 
\newtheorem*{rems}{Remarks} 
\newtheorem*{ack}{Acknowledgments} 
\newtheorem*{note}{Note}
\newtheorem{claimn}{Claim}
\newtheorem{subclaim}{Subclaim}
\newtheorem*{subclaimnn}{Subclaim}
\newtheorem*{subclaim1}{Subclaim (i)}
\newtheorem*{subclaim2}{Subclaim (ii)}
\newtheorem*{subclaim3}{Subclaim (iii)}
\newtheorem*{subclaim4}{Subclaim (iv)}
\newtheorem{case}{Case}
 \newtheorem*{case1}{Case 1}
\newtheorem*{case2}{Case 2}
\newtheorem*{case3}{Case 3}
\newtheorem*{case4}{Case 4}

\newcommand{\noprint}[1]{\relax}
\newenvironment{nd}{\noindent\color{red}ND: }{}
\newenvironment{jsm}{\noindent\color{blue}SM: }{}

\title[Infinite dimensional Ellentuck spaces]{Infinite dimensional  Ellentuck spaces and Ramsey-classification theorems}
\author{Natasha Dobrinen}
\address{Department of Mathematics\\
  University of Denver \\
   2280 S Vine St\\ Denver, CO \ 80208 U.S.A.}
\email{natasha.dobrinen@du.edu}
\urladdr{\url{http://web.cs.du.edu/~ndobrine}} 
\thanks{This work  was  supported by
 National Science Foundation Grant DMS-1301665}



\begin{abstract}
We extend the hierarchy of finite-dimensional Ellentuck spaces  to infinite dimensions.
Using uniform barriers $B$ on $\omega$ as the prototype structures, we construct  a class of continuum many topological Ramsey spaces $\mathcal{E}_B$ which are Ellentuck-like in nature,
and form a linearly ordered hierarchy under projection.
We prove new Ramsey-classification theorems for equivalence relations on fronts, and hence also on barriers, on the spaces $\mathcal{E}_B$, extending the \Pudlak-\Rodl\ Theorem for barriers on the Ellentuck space.

The inspiration for these spaces comes from  continuing the iterative construction of the forcings $\mathcal{P}([\om]^k)/\Fin^{\otimes k}$  to the countable transfinite.
The $\sigma$-closed partial order  $(\mathcal{E}_B, \sse^{\Fin^{B}})$ is  forcing equivalent to 
$\mathcal{P}(B)/\Fin^{B}$,
 which forces a non-p-point  ultrafilter $\mathcal{G}_B$. 
 The present work forms the basis for 
further work classifying the Rudin-Keisler and Tukey structures for  the hierarchy of the generic ultrafilters $\mathcal{G}_B$.
\end{abstract}


\maketitle


\section{Overview}\label{sec.overview}

Extending work  in \cite{DobrinenJSL15}, we construct  a new class of topological Ramsey spaces which are structurally based on uniform barriers.
The inspiration for these  spaces comes from the continuation of the iterative construction of the Boolean algebras  $\mathcal{P}(\om^k)/\Fin^{\otimes k}$ ($1\le k<\om$) to the countable transfinite.
These spaces form a natural hierarchy in complexity over the Ellentuck space and the finite-dimensional Ellentuck spaces of  \cite{DobrinenJSL15} from several viewpoints.
First,
 whenever $B,C$ are uniform barriers on $\om$ with rank of $C$ less than rank of $B$, 
then $\mathcal{E}_B$  projects to the space $\mathcal{E}_C$.
Second,
any restriction of  the members of a space $\mathcal{E}_B$ to the extensions of some fixed finite initial segment in some member of the space yields an
 isomorphic copy of one of the  lower-dimensional Ellentuck spaces that were used to form $\mathcal{E}_B$.
This is one of the justifications of the terminology
infinite-dimensional Ellentuck space.
Third, for any uniform barriers $B$ and $C$, one of  $\mathcal{E}_B$ and $\mathcal{E}_C$  projects to the other,  modulo possibly restricting below some $X\in\mathcal{E}_B$ and some $Y\in\mathcal{E}_C$.

The three main contributions of this work are the construction of the spaces $\mathcal{E}_B$, the proof that they are indeed topological Ramsey spaces, and the Ramsey-classification Theorem canonizing equivalence relations on fronts on $\mathcal{E}_B$ in terms of projection maps to subtrees which are unique among projection maps satisfying a certain natural property. 
The  Ramsey-classification Theorem can be seen as the natural extension of the \Pudlak-\Rodl\ Theorem in \cite{Pudlak/Rodl82} to these spaces, the \Pudlak-\Rodl\ Theorem, and moreover the Ramsey-classification Theorems for the finite-dimensional Ellentuck spaces in \cite{DobrinenJSL15}, being recovered from our theorem via a projection map.

\section{Introduction}\label{sec.intro}

The Infinite Ramsey Theorem  states that given any positive integers $k$ and  $l$ and any coloring $c$ of  the subsets of the natural numbers of size $k$ into $l$ many colors, there is an infinite set $M$ of natural numbers such that 
 $c$ is constant on   the subsets of $M$ of size $k$  (see \cite{Ramsey29}).
Nash-Williams extended this  to finite colorings of  any  barrier  in \cite{NashWilliams65}.
A finite coloring of a barrier  codes a partition of the Baire space into finitely many clopen sets.
Galvin and Prikry extended the Nash-Williams Theorem to Borel partitions of the Baire space in \cite{Galvin/Prikry73}, and Silver extended it to analytic partitions in \cite{Silver70}, though by non-constructive methods.

Shortly thereafter, Ellentuck  found the optimal Ramsey theorem for finite partitions of the Baire space. 
He captured   the essence of the Ramsey property 
by  equipping
 the Baire space with a  topology finer than the metric topology,  now  called the Ellentuck topology.
In \cite{Ellentuck74},  Ellentuck proved  
 that a subset of the Baire space is  Ramsey  (see Definition \ref{defn.5.2}) if and only if it has the property of Baire in the Ellentuck topology, thus 
providing a constructive proof of Silver's Theorem and moreover 
 extending extending it to all  partitions of the Baire space into finitely many sets each of which has the property of Baire in the Ellentuck topology.

Independently and around the same time,  Louveau   developed ultra-Ramsey methods using a topology similar to Ellentuck's but with infinite sets in a fixed Ramsey ultrafilter, and 
proved the analogue of Ellentuck's Theorem in the ultra-Ramsey setting (see \cite{Louveau74}).
His work  provided a topological proof of Silver's Theorem 
 a theorem of Mathias in \cite{Mathias77} in the ultra-Ramsey setting.
An ultrafilter is {\em Ramsey} if for any finite coloring of the pairs of natural numbers, there is a member of the ultrafilter on which the coloring is homogeneous.
Ramsey ultrafilters are closely connected with the Ellentuck space.
Assuming some axiom in addition to the standard axioms of set theory, for instance the Continuum Hypothesis, Martin's Axiom, or even less ($\mathfrak{p}=\mathfrak{c}$), or by the method of forcing,
one can construct a Ramsey ultrafilter using the Boolean algebra $\mathcal{P}(\om)/\Fin$.
Equivalently, one may use the Ellentuck space partially ordered by almost inclusion.
This sets the stage for later connections between topological Ramsey spaces and ultrafilters satisfying partition properties.

The Ellentuck space is the quintessential example of the more general notion of  a topological Ramsey space (see Definition \ref{defn.5.2}).
Classic examples of topological Ramsey spaces include the Carlson-Simpson space in \cite{Carlson/Simpson84} of equivalence relations on the natural numbers with infinitely many equivalence classes, and the Milliken space of infinite block sequences (see \cite{Milliken75})  which has proved fundamental for progress in certain areas of Banach space theory.
Building on prior work of Carlson and Simpson,
Todorcevic distilled four axioms which, when satisfied, guarantee that a space is a toplogical Ramsey space (see Section \ref{sec.reviewtRs}).
This axiomatic approach to topological Ramsey spaces paved the way for recent work involving connections between ultrafilters satisfying partition properties and constructions of  new topological Ramsey spaces (see \cite{Dobrinen/Todorcevic14}, \cite{Dobrinen/Todorcevic15},
\cite{Dobrinen/Mijares/Trujillo14}, and \cite{DobrinenJSL15}).

Similar to the connection between the Ellentuck space and Ramsey ultrafilters,  each topological Ramsey space has associated ultrafilters, which 
are `selective' or `Ramsey' with respect to the space (see \cite{Mijares07} and  \cite{Trujillo15}).
Using the Continuum Hypothesis or Martin's Axiom (indeed, less is actually required) or forcing, any topological Ramsey space gives rise to an associated ultrafilter satisfying some partition properties, dependent on the space.

One current research program of ours  involves finding the essential Ramsey-like structures inside  partial orderings which are responsible for the partition properties of the ultrafilters that they construct.
The motivation is as follows.
Given a partial order which  constructs an ultrafilter with certain partition properties, if one can  show that  the partial order contains a topological Ramsey space as a dense subset, then 
one gains much machinery and a fine-tuned approach to investigations of 
 the properties of the ultrafilter.
To begin with,  each topological Ramsey space satisfies its version of the Abstract Ellentuck Theorem and the Abstract Nash-Williams Theorem 
(see Theorems \ref{thm.AET} and \ref{thm.ANW}).
The structure of the members of the topological Ramsey space  provide insight  and these theorems streamline proofs involving 
 finding the Ramsey numbers for the associated ultrafilters.
Further,  the structure
 sets the stage for  canonizing equivalence relations on fronts and barriers,  yielding analogues of the \Pudlak-\Rodl\ theorem  in \cite{Pudlak/Rodl82}.
This in turn makes it possible to find the exact Rudin-Keisler and Tukey structures below the ultrafilters.

Todorcevic was the first to notice the connection between canonical equivalence relations on fronts and the Tukey type of a Ramsey ultrafilter.
Tukey reduction on ultrafilters  is a generalization of Rudin-Keisler reduction 
and has been widely studied in \cite{Milovich08}, \cite{Dobrinen/Todorcevic11}, \cite{Milovich12}, \cite{Raghavan/Todorcevic12},
\cite{Dobrinen/Todorcevic14}, \cite{Dobrinen/Todorcevic15},
\cite{Dobrinen/Mijares/Trujillo14}, 
\cite{Blass/Dobrinen/Raghavan15}, and \cite{DobrinenJSL15}.
An ultrafilter $\mathcal{V}$ is {\em Tukey reducible} to an ultrafilter $\mathcal{U}$ if and only if there is a function $f:\mathcal{U}\ra\mathcal{V}$ which sends  each filter base of $\mathcal{U}$ to a filter base of $\mathcal{V}$.
Since the Tukey theory of ultrafilters  is not the focus of the results in this paper, we merely mention its strong motivating force in the current work without going into detail.
For more background on Tukey theory of ultrafilters, the reader is referred to our survey paper \cite{DobrinenTukeySurvey15}.
In \cite{Raghavan/Todorcevic12}, Todorcevic proved that the Tukey type of any Ramsey ultrafilter is minimal  via a judicious application of
 the \Pudlak-\Rodl\ Theorem on the Ellentuck space.
In \cite{Dobrinen/Todorcevic14} and \cite{Dobrinen/Todorcevic15}, Dobrinen and Todorcevic constructed a new class of  topological Ramsey spaces which are dense in the partial orderings  of Laflamme  in \cite{Laflamme89} for ultrafilters with weaker partition properties.  By proving and applying Ramsey-classification theorems, they found the exact initial  Rudin-Keisler and Tukey structures for this class of ultrafilters.
In \cite{Dobrinen/Mijares/Trujillo14}, Dobrinen, Mijares, and Trujillo developed a template for constructions of topological Ramsey spaces, including spaces which generate the $k$-arrow ultrafilters of Baumgartner and Taylor in \cite{Baumgartner/Taylor78}; they proved broad class of Ramsey-classification theorems and applied them to find exact Rudin-Keisler and Tukey structures of the associated ultrafilters.
Each of the new topological Ramsey spaces mentioned above has only finitely many components extending each node, and thus  cannot be thought of as 
a generalizing the Ellentuck space.

The work in this paper builds a new class of continuum many  infinite dimensional Ellentuck spaces.  
This extends our previous work in \cite{DobrinenJSL15} on finite dimensional Ellentuck spaces.
That paper was motivated by 
  \cite{Blass/Dobrinen/Raghavan15},  in which Blass, Dobrinen, and Raghavan initiated the study of 
the Tukey type of the generic ultrafilter $\mathcal{G}_2$ forced by $\mathcal{P}(\om\times\om)/\Fin\otimes\Fin$, but did not find the exact Tukey structure below $\mathcal{G}_2$. 
(Hrusak and Verner also considered this forcing in  \cite{Hrusak/Verner11}.)
In \cite{DobrinenJSL15}, the author developed a hierarchy of spaces $\mathcal{E}_k$ extending the Ellentuck space to every finite dimension $k$.
The  space $\mathcal{E}_k$  forms  dense subsets of the partial order $\mathcal{P}(\om^k)/\Fin^{\otimes k}$   distilling the  internal  structure  responsible for its  forcing and Ramsey properties.
We proved Ramsey-classification theorems for the $\mathcal{E}_k$ spaces and applied them to find that indeed, the initial Rudin-Keisler and Tukey structures below $\mathcal{G}_k$ is  simply  a linear order of size $k$ (see Theorems 40 and 41 in \cite{DobrinenJSL15}).

The inspiration for the the infinite-dimensional Ellentuck  spaces comes from the continuation of the iterative construction of the Boolean algebras  $\mathcal{P}(\om^k)/\Fin^{\otimes k}$ to the countable transfinite.
The underlying structure of these spaces $\mathcal{E}_B$ are based on   uniform barriers on $\om$, which also give the rank of the space.
These spaces form a natural hierarchy in complexity over the Ellentuck space and the hierarchy of finite dimensional Ellentuck spaces in \cite{DobrinenJSL15} space from several viewpoints.
First,
 whenever $B,C$ are uniform barriers on $\om$ with rank of $C$ less than rank of $B$, 
then $\mathcal{E}_B$ projects to $\mathcal{E}_C$.
Second,
as seen in Fact \ref{fact.important}, if we 
fix any finite initial segment of an element of a  member of the space and restrict the space to those members that extend it, then we obtain an isomorphic copy of one of the  lower-dimensional Ellentuck spaces that was used to form $\mathcal{E}_B$.
Lastly, for any uniform barriers $B$ and $C$,  we can always find members $X\in\mathcal{E}_B$ and $Y\in\mathcal{E}_C$ so that 
one of  $\mathcal{E}_B|X$ and $\mathcal{E}_Y$ projects to the other (see Fact \ref{fact.EBCemb}).

Associated to each $\mathcal{E}_B$ is the 
  $\sigma$-closed forcing $\mathcal{P}(B)/\Fin^{B}$, which forces a non-p-point  ultrafilter $\mathcal{G}_B$ (see just below Fact \ref{fact.2barrierscompare} in Section \ref{sec.FinB}).  
Like the spaces $\mathcal{E}_B$, these ultrafilters form a hierarchy, in that whenever $C$ is obtained by a projection of $B$, then $\mathcal{G}_C$ is Rudin-Keisler reducible to $\mathcal{G}_B$.
The Ramsey-classification Theorem \ref{thm.rc} canonizing equivalence relations on fronts on the spaces $\mathcal{E}_B$ will be applied
in forthcoming work to find the Rudin-Keisler and Tukey structures below the ultrafilters $\mathcal{G}_B$.  
For now, we will only  mention that, for infinite rank  $B$, we have proved that there are continuum sized linear orders in  both the Rudin-Keisler and  Tukey structures of ultrafilters   reducible to $\mathcal{G}_B$.
This  contrasts with all  previous related works which only yielded countable initial Tukey structures.

The paper is organized as follows.
Section \ref{sec.reviewtRs} provides the necessary  background on topological Ramsey spaces from Todorcevic's book \cite{TodorcevicBK10} as well as a review of the Ellentuck space and the \Pudlak-\Rodl\ Theorem.
Section \ref{sec.FinB} begins with a review of uniform barriers and then carries out the recursive construction of the $\sigma$-closed ideals $\Fin^B$, $B$ a uniform barrier, and the Boolean algebras $\mathcal{P}(B)/\Fin^B$.
This section also includes a brief discussion of the ultrafilters constructed by $\mathcal{P}(B)/\Fin^B$ and  some basic facts about their Tukey and Rudin-Keisler structures, forming the backdrop to a sequel paper.

The   infinite dimensional Ellentuck spaces $\mathcal{E}_B$   are defined in 
Definition \ref{def.E_B}  in
 Section \ref{sec.defR}. 
These spaces are constructed recursively on the rank of the barrier $B$, via 
defining the legitimate domains (Definitions \ref{defn.prec} and \ref{defn.domainspace}) so as to ensure that spaces of smaller rank embed into spaces of larger rank, and in a way that the space restricted above any fixed stem is isomorphic to a space of smaller rank (see Fact \ref{fact.important}).
In one of the main theorems of this paper,  Theorem \ref{thm.E_BtRs}, we prove that for each uniform barrier $B$, the space $\mathcal{E}_B$ is a topological Ramsey space.

In Section \ref{sec.canoneqrelAE1}, 
Theorems \ref{thm.CanonicalonAE_1} and \ref{thm.canonAR1}
prove the existence of canonical equivalence relations on first approximations and on 1-extensions.  These are given in terms of uniform projections, which are similar in  form to uniform barriers and have a notion of rank.
Thus, there are continuum many canonical equivalence relations on first approximations,
in contrast to the finite-dimensional Ellentuck spaces which have only finitely many such canonical equivalence relations.
These theorems  
lay the groundwork for the Ramsey-classification theorems in the next section.

In Section \ref{sec.rct} we prove  the  Ramsey-classification Theorem \ref{thm.rc}
for each $\mathcal{E}_B$.
In words, it says that each equivalence relation on a front $\mathcal{F}$ on the space $\mathcal{E}_B$ is canonized by an inner Nash-Williams  function  which projects each member $u$ of $\mathcal{F}$ to some subtree of $u$.
Further, this projection function   uniquely canonizes the equivalence relation
among those inner functions which satisfy 
the property $(*)$ (see Definition \ref{def.innerNW}
and Theorem \ref{thm.irred}).
 Though the proof of Theorem \ref{thm.rc} follows the standard outline of the analogous theorems in \cite{Dobrinen/Todorcevic14}, \cite{Dobrinen/Todorcevic15}, \cite{Dobrinen/Mijares/Trujillo14}, and \cite{DobrinenJSL15},
the proof methods are  quite different, involving  new ideas central to the structure of the spaces $\mathcal{E}_B$.


\section{Basics of   topological Ramsey spaces}\label{sec.reviewtRs}

A brief review of topological Ramsey spaces is provided in this section for the reader's convenience.
Building on prior work of Carlson and Simpson in \cite{Carlson/Simpson90}, Todorcevic distilled  key properties of the Ellentuck space into  four axioms, \bf A.1  \rm -  \bf A.4\rm, which guarantee that a space is a topological Ramsey space.
(For further background, the reader  is referred to Chapter 5 of \cite{TodorcevicBK10}.)
The  axioms \bf A.1  \rm -  \bf A.4 \rm
are defined for triples
$(\mathcal{R},\le,r)$
of objects with the following properties:
$\mathcal{R}$ is a nonempty set,
$\le$ is a quasi-ordering on $\mathcal{R}$,
 and $r:\mathcal{R}\times\om\ra\mathcal{AR}$ is a  map producing the sequence $(r_n(\cdot)=r(\cdot,n))$ of  restriction maps, where
$\mathcal{AR}$ is  the collection of all finite approximations to members of $\mathcal{R}$.
For $u\in\mathcal{AR}$ and $X,Y\in\mathcal{R}$,
\begin{equation}
[u,X]=\{Y\in\mathcal{R}:Y\le X\mathrm{\ and\ }(\exists n)\ r_n(Y)=u\}.
\end{equation}

For $u\in\mathcal{AR}$, let $|u|$ denote the length of the sequence $u$.  Thus, $|u|$ equals the integer $k$ for which $u=r_k(u)$.
For $u,v\in\mathcal{AR}$, $u\sqsubseteq v$ if and only if $u=r_m(v)$ for some $m\le |v|$.
$u\sqsubset v$ if and only if $u=r_m(v)$ for some $m<|v|$.
For each $n<\om$, $\mathcal{AR}_n=\{r_n(X):X\in\mathcal{R}\}$.
\vskip.1in

\begin{enumerate}
\item[\bf A.1]\rm
\begin{enumerate}
\item[(1)]
$r_0(X)=\emptyset$ for all $X\in\mathcal{R}$.\vskip.05in
\item[(2)]
$X\ne Y$ implies $r_n(X)\ne r_n(Y)$ for some $n$.\vskip.05in
\item[(3)]
$r_m(X)=r_n(Y)$ implies $m=n$ and $r_k(X)=r_k(X)$ for all $k<n$.\vskip.1in
\end{enumerate}
\item[\bf A.2]\rm
There is a quasi-ordering $\le_{\mathrm{fin}}$ on $\mathcal{AR}$ such that\vskip.05in
\begin{enumerate}
\item[(1)]
$\{v\in\mathcal{AR}:v\le_{\mathrm{fin}} u\}$ is finite for all $u\in\mathcal{AR}$,\vskip.05in
\item[(2)]
$Y\le X$ iff $(\forall n)(\exists m)\ r_n(Y)\le_{\mathrm{fin}} r_m(X)$,\vskip.05in
\item[(3)]
$\forall u,v,y\in\mathcal{AR}[y\sqsubset v\wedge v\le_{\mathrm{fin}} u\ra\exists x\sqsubset u\ (y\le_{\mathrm{fin}} x)]$.\vskip.1in
\end{enumerate}
\end{enumerate}

The number $\depth_X(u)$ is the least $n$, if it exists, such that $u\le_{\mathrm{fin}}r_n(X)$.
If such an $n$ does not exist, then we write $\depth_X(u)=\infty$.
If $\depth_X(u)=n<\infty$, then $[\depth_X(u),X]$ denotes $[r_n(X),X]$.

\begin{enumerate}
\item[\bf A.3] \rm
\begin{enumerate}
\item[(1)]
If $\depth_X(u)<\infty$ then $[u,Y]\ne\emptyset$ for all $Y\in[\depth_X(u),X]$.\vskip.05in
\item[(2)]
$Y\le X$ and $[u,Y]\ne\emptyset$ imply that there is $Y'\in[\depth_X(u),X]$ such that $\emptyset\ne[u,Y']\sse[u,Y]$.\vskip.1in
\end{enumerate}
\end{enumerate}

Additionally, 
for $n>|u|$, let  $r_n[u,X]$  denote the collection of all $v\in\mathcal{AR}_n$ such that $u\sqsubset v$ and $v\le_{\mathrm{fin}} X$.

\begin{enumerate}
\item[\bf A.4]\rm
If $\depth_X(u)<\infty$ and if $\mathcal{O}\sse\mathcal{AR}_{|u|+1}$,
then there is $Y\in[\depth_X(u),X]$ such that
$r_{|u|+1}[u,Y]\sse\mathcal{O}$ or $r_{|u|+1}[u,Y]\sse\mathcal{O}^c$.\vskip.1in
\end{enumerate}

The  {\em Ellentuck topology} on $\mathcal{R}$ is the topology generated by the basic open sets
$[u,X]$;
it refines the  metric topology on $\mathcal{R}$,  considered as a subspace of the Tychonoff cube $\mathcal{AR}^{\bN}$.
Given the Ellentuck topology on $\mathcal{R}$,
the notions of nowhere dense, and hence of meager are defined in the natural way.
We  say that a subset $\mathcal{X}$ of $\mathcal{R}$ has the {\em property of Baire} iff $\mathcal{X}=\mathcal{O}\cap\mathcal{M}$ for some Ellentuck open set $\mathcal{O}\sse\mathcal{R}$ and Ellentuck meager set $\mathcal{M}\sse\mathcal{R}$.

\begin{defn}[\cite{TodorcevicBK10}]\label{defn.5.2}
A subset $\mathcal{X}$ of $\mathcal{R}$ is {\em Ramsey} if for every $\emptyset\ne[u,X]$,
there is a $Y\in[u,X]$ such that $[u,Y]\sse\mathcal{X}$ or $[u,Y]\cap\mathcal{X}=\emptyset$.
$\mathcal{X}\sse\mathcal{R}$ is {\em Ramsey null} if for every $\emptyset\ne [u,X]$, there is a $Y\in[u,X]$ such that $[u,Y]\cap\mathcal{X}=\emptyset$.

A triple $(\mathcal{R},\le,r)$ is a {\em topological Ramsey space} if every subset of $\mathcal{R}$  with the property of Baire  is Ramsey and if every meager subset of $\mathcal{R}$ is Ramsey null.
\end{defn}

The following result can be found as Theorem
5.4 in \cite{TodorcevicBK10}.

\begin{thm}[Abstract Ellentuck Theorem]\label{thm.AET}\rm \it
If $(\mathcal{R},\le,r)$ is closed (as a subspace of $\mathcal{AR}^{\bN}$) and satisfies axioms {\bf A.1}, {\bf A.2}, {\bf A.3}, and {\bf A.4},
then every  subset of $\mathcal{R}$ with the property of Baire is Ramsey,
and every meager subset is Ramsey null;
in other words,
the triple $(\mathcal{R},\le,r)$ forms a topological Ramsey space.
\end{thm}

\begin{defn}[\cite{TodorcevicBK10}]\label{defn.5.16}
A family $\mathcal{F}\sse\mathcal{AR}$ of finite approximations is
\begin{enumerate}
\item
{\em Nash-Williams} if  $u\ne v\in\mathcal{F}$ implies $v\not\sqsubseteq u$;
\item
{\em Ramsey} if for every partition $\mathcal{F}=\mathcal{F}_0\cup\mathcal{F}_1$ and every $X\in\mathcal{R}$,
there are $Y\le X$ and $i\in\{0,1\}$ such that $\mathcal{F}_i|Y=\emptyset$.
\end{enumerate}
\end{defn}

The Abstract Nash-Williams Theorem (Theorem 5.17  in \cite{TodorcevicBK10}), which follows from the Abstract Ellentuck Theorem,
is actually  sufficient for the proofs in Sections \ref{sec.canoneqrelAE1} and \ref{sec.rct}.

\begin{thm}[Abstract Nash-Williams Theorem]\label{thm.ANW}
Suppose $(\mathcal{R},\le,r)$ is a closed triple that satisfies {\bf A.1} - {\bf A.4}. Then
every Nash-Williams family of finite approximations is Ramsey.
\end{thm}

\begin{defn}\label{def.frontR1}
Suppose $(\mathcal{R},\le,r)$ is a closed triple that satisfies {\bf A.1} - {\bf A.4}.
Let $X\in\mathcal{R}$.
A family $\mathcal{F}\sse\mathcal{AR}$ is a {\em front} on $[0,X]$ if
\begin{enumerate}
\item
For each $Y\in[0,X]$, there is a $u\in \mathcal{F}$ such that $u\sqsubset Y$; and
\item
$\mathcal{F}$ is Nash-Williams.
\end{enumerate}
\end{defn}

\begin{rem}
The  stronger notion of  {\em barrier} on a topological Ramsey space  simply replaces Nash-Williams with {\em Sperner} in (2) of Definition \ref{def.frontR1}: 
whenever $u,v\in\mathcal{F}$, $v\ne u\ra v\not\le_{\fin} u$.
For  topological Ramsey spaces 
in which the  quasi-order $\le_{\fin}$ is actually a partial order, 
each front is a barrier when restricted to some member of $\mathcal{R}$ 
(see  Corollary 5.19 in \cite{TodorcevicBK10}).
Since $\le_{\fin}$ is a partial order for all the spaces in this article,
the terms `front' and `barrier' may be used interchangeably  with no change  to our results.
\end{rem}

This section concludes by recalling the Ellentuck space and stating  the \Pudlak-\Rodl\ Theorem for canonical equivalence relations on fronts on the Ellentuck space.

\begin{defn}\label{def.Ellspace}
Let $(\mathcal{E}, \le,r)$ denote the {\em Ellentuck space} $([\om]^{\om},\sse,r)$,
where for each $X\in[\om]^{\om}$ and $n<\om$, 
$r_n(X)$ denotes the set of the $n$ least members of $X$.
\end{defn}

\begin{defn}\label{def.irredPR}
A map $\vp$ on a front $\mathcal{F}\sse [\om]^{<\om}$ on the Ellentuck space is called
\begin{enumerate}
\item
{\em inner} if for each $u\in \mathcal{F}$, 
$\vp(u)\sse u$.
\item
{\em Nash-Williams} if for all pairs $u,v\in \mathcal{F}$,
$\vp(u)\not\sqsubset \vp(v)$.
\item 
{\em irreducible} if it is inner and Nash-Williams.
\end{enumerate}
\end{defn}

\begin{thm}[\Pudlak/\Rodl,  \cite{Pudlak/Rodl82}]
Let $R$ be an equivalence relation on a front $\mathcal{F}$ on the Ellentuck space.
Then there is an irreducible map $\vp$ and an $X\in[\om]^{\om}$ such that for all $u,v\in\mathcal{F}$ with $u,v\sse X$,
\begin{equation}
u\, R\, v\ \ \llra\ \ \vp(u)=\vp(v).
\end{equation}
\end{thm}

This  theorem has been generalized to new classes of  topological Ramsey spaces in the papers \cite{Dobrinen/Todorcevic14},  \cite{Dobrinen/Todorcevic15},  \cite{Dobrinen/Mijares/Trujillo14} and \cite{DobrinenJSL15}.
In Section \ref{sec.rct}, we will extend it to the infinite dimensional Ellentuck spaces, which will be defined in Section \ref{sec.defR}.


\section{Uniform barriers $B$, the Boolean algebras $\mathcal{P}(B)/\Fin^B$, and the generic ultrafilters $\mathcal{G}_B$}\label{sec.FinB}

This section provides  background on uniform barriers $B$ on $\om$ and our  recursive  constructions of the $\sigma$-closed ideals $\Fin^B$ on $\mathcal{P}(B)$.
These ideals are used to
extend the class of  Boolean algebras $\mathcal{P}([\om]^k)/\Fin^{\otimes k}$ in \cite{DobrinenJSL15} to 
the class of  Boolean algebras 
$\mathcal{P}(B)/\Fin^B$, for every  uniform barrier $B$ of countable ordinal rank.
Since each such Boolean algbebra has a $\sigma$-closed dense subset, forcing with it adds no new subsets of $\om$ or of $B$.
Using the Continuum Hypthothesis, Martin's Axiom, or by forcing, 
$\mathcal{P}(B)/\Fin^B$ constructs 
 a non-p-point ultrafilter, which we shall denote $\mathcal{G}_B$, satisfying some partition properties.
Investigations of the Rudin-Keisler and Tukey structures of these generic ultrafilters  motiviated the construction of the topological Ramsey spaces $\mathcal{E}_B$ and their Ramsey-classification Theorems, which will be applied to classify
their  Rudin-Keisler and Tukey structures in forthcoming work.

For $a,b\in[\om]^{<\om}$,
we shall use the notation $a\trianglelefteq b$  to denote that $a$ is an initial segment of $b$, and $a\vartriangleleft b$ to denote that $a$ is a proper initial segment of $b$.  
This will serve to 
 distinguish the partial ordering of  initial segment on $[\om]^{<\om}$ from the  partial ordering $u\sqsubseteq  v$ of  initial segment
for $u,v\in\mathcal{AR}$,
  finite approximations of members of a topological Ramsey space $\mathcal{R}$.

The following two definitions can be found in \cite{Argyros/TodorcevicBK}.
The notation is  slightly modified to be more suitable  for the present work.

\begin{defn}\label{defn.barrier}
A family $B$ of finite subsets of $\om$ is a 
{\em front} on $\bigcup B$ if
\begin{enumerate}
\item
$a\not \tle b$ whenever $a\ne b\in B$; and 
\item
$\bigcup B$ is infinite and for every infinite $M\sse\bigcup B$ there is an $a\in B$ such that $a\vartriangleleft M$.
\end{enumerate}
$B$ is a 
{\em barrier} on $\bigcup B$
if
\begin{enumerate}
\item[$(1')$]
$a\not\sse b$ whenever $a\ne b\in B$; and
\end{enumerate}
(2) holds.
\end{defn}

\begin{notation}\label{notn.B_n}
For a barrier $B$ and $n\in\om$, let $B_n=\{b\in B:n=\min(b)\}$;
and let $B_{\{n\}}=\{a\in[\om]^{<\om}:\min(a)>n$ and $\{n\}\cup a\in B\}$.
For $N$ an infinite subset of $\bigcup B$,
$B| N=\{b\in B:b\sse N\}$.
\end{notation}

\begin{defn}\label{defn.uniformbarrier}
Let $\al<\om_1$ and $M$ be an infinite subset of $\om$.
A  subset $B\sse[\om]^{<\om}$ is an {\em $\al$-uniform family on $M$} provided that
\begin{enumerate}
\item[(a)]
$\al=0$  implies $B=\emptyset$.
\item[(b)]
$\al=\beta+1$ implies that $\emptyset\not\in B$ and $B_{\{n\}}$ is $\beta$-uniform on $M\setminus (n+1)$, for all $n\in M$.
\item[(c)]
$\al>0$ is a limit ordinal implies that there is an increasing sequence $\{\al_n\}_{n\in M}$ of ordinals converging to $\al$ such that $B_{\{n\}}$ is $\al_n$-uniform on $M\setminus (n+1)$, for all $n\in M$.
\end{enumerate}
A barrier $B\sse[\om]^{<\om}$ which is also a uniform family is called a {\em uniform barrier}.
\end{defn}

The following facts show that, were we to use fronts instead of uniform barriers
 as bases for the structures  $\Fin^B$ and $\mathcal{E}_B$, we would gain nothing more. 
The following results stated in Fact \ref{facts.barriersencompassfronts} appear as 
Lemmas II.3.2, II.3.3,  II.3.8 (due to Galvin),  Corollary II.3.10 and  Lemma II.3.17, respectively, in  \cite{Argyros/TodorcevicBK}.

\begin{facts} 
 \label{facts.barriersencompassfronts}
\begin{enumerate}
\item
Every $\al$-uniform family on $M$ is a front on $M$.
\item
If $B$ is $\al$-uniform on $M$ and if $N$ is an infinite subset of $M$, then $B| N$ is $\al$-uniform on $N$.
\item
For every family $B\sse[\om]^{<\om}$ and every infinite $M\sse \om$, there is an infinite $N\sse M$ such that the restriction $B| N$ is either empty or contains a barrier.
\item
For every family $B\sse[\om]^{<\om}$ and for every infinite $M\sse \om$, there is an infinite $N\sse M$ such that  either
$B| N=\emptyset$ or  $B| N$ contains a uniform barrier.
\item
The lexicographical rank of an $\al$-uniform barrier is equal to $\om^{\al}$.
\end{enumerate}
\end{facts}

Thus, for every front $B\sse[\om]^{<\om}$ there is an infinite $N\sse\bigcup B$ such that  $B| N$ is a uniform barrier, and any further restriction $B| N'$ for $N'\in[N]^{\om}$ is again a uniform barrier with the same rank as $B| N$.
Since the lexicographic rank of an $\al$-uniform barrier is $\om^{\al}$, we will  simply  say that the barrier is  {\em uniform of  rank $\al$}.

\begin{examples}
$\{\emptyset\}$ is the uniform barrier of rank $0$.
For each $1\le k<\om$, $[\om]^k$ is the uniform barrier  of rank $k$.
For $\om\le \al<\om_1$ there are a myriad of uniform barriers of rank $\al$. 
The most well-known is the {\em Schreier barrier}
$S=\{b\in[\om]^{<\om}: |s|=\min(s)+1\}$, which is uniform of rank $\om$.
\end{examples}

For $B,C\sse[\om]^{<\om}$,
write $C\tl B$ if and only if for each $c\in C$  there is a  $b\in B$  such that $c\tl b$.
The next fact is an immediate consequence of the following:  Every barrier on $\om$ is a Nash-Williams family on the Ellentuck space, and hence is 
Ramsey (recall Definition \ref{defn.5.16}). 
Given barriers $B$ and $C$ on $\om$, we may color each $c\in C$ by $0$ if there is a $b\in B$ such that $b= c$, $1$ if there is a $b\in B$ such that $b\tl c$, and $2$ if there is a $b\in B$ such that $b\tr c$.
Applying the Nash-Williams Theorem for this $3$-coloring on $C$ yields the following.
In the case of $B|M\tl C|M$, we shall call $B|M$ a {\em projection} of $C|M$.

\begin{fact}\label{fact.2barrierscompare}
Suppose $B$ and $C$ are two barriers  on the same infinite set $N\sse \om$.
Then there is an infinite $M\sse N$ such that one of the following hold:
$B|M=C|M$, $B|M\tl C|M$, or $B|M\tr C|M$.
\end{fact}

Now we develop the hierarchy of ideals  $\Fin^B$  on $\mathcal{P}(B)$, for $B$ a uniform barrier.
Recall that $\Fin$ denotes the collection of finite subsets of $\om$.
In \cite{DobrinenJSL15}, we let $\Fin^{\otimes 1}$ denote $\Fin$, and
 given $\Fin^{\otimes k}$, $k<\om$, we defined $\Fin^{\otimes k+1}$ to be  the collection of all $A\sse[\om]^{k+1}$ such that for all but finitely many $a_0\in\om$, the set $\{\{a_1,\dots,a_{k}\}\in[\om]^k:a_1>a_0$ and $\{a_0,a_1,\dots,a_{k}\}\in A\}$ is a member of $\Fin^{\otimes k}$.
In the present paper, we shall use the simpler notation
 $\Fin^B$, for $B$ a uniform barrier.
Then $\Fin^{[\om]^k}$ denotes $\Fin^{\otimes k}$  from \cite{DobrinenJSL15}, for  $k<\om$.
Given a uniform barrier $B$ on $\om$ of rank $\al$,  note that
 $B_{\{n\}}$ is a uniform barrier on $\om\setminus (n+1)$ of rank less than $\al$, for any $n<\om$.
Assuming we have defined $\Fin^C$ for each uniform barrier $C$ of rank less than $\al$,
then we have also  defined $\Fin^{B_{\{n\}}}$ for each  $n<\om$
(relativizing to $\om\setminus n+1$).
Let $\Fin^{B_n}=\{X\sse B_n:\{a\setminus \{n\}:a\in X\}\in\Fin^{B_{\{n\}}}\}$.
\begin{defn}\label{def.FinB}
For $B$ a uniform barrier, define 
\begin{equation}
\Fin^B=\{A\sse B:\forall^{\infty} n\ ( A_n\in\Fin^{B_n})\}.
\end{equation}
\end{defn}

For each uniform barrier $B$ on $\om$,  $\Fin^B$ is a $\sigma$-closed ideal; hence,
$(\mathcal{P}(B)/\Fin^B)\setminus \{\bo\}$ 
is a $\sigma$-closed partial order which can be used to construct an  ultrafilter $\mathcal{G}_B$ (on the countable base set $B$).
If $B$ and $C$ are uniform barriers with $C\tl B$, 
the projection map $\pi_{B,C}$ from $B$ to $C$ induces a projection from  $\mathcal{P}(B)/\Fin^B$  to $\mathcal{P}(C)/\Fin^C$, and hence induces a projection from the ultrafilter $\mathcal{G}_B$ to $\mathcal{G}_C$.
It follows that $\pi_{B,C}(\mathcal{G}_B)$ is generic for $\mathcal{P}(C)/\Fin^C$.
Thus, the generic ultrafilters form a hierarchy in the isomorphism classes of ultrafilters:  If $C$ is a projection of $B$, then $\mathcal{G}_C$ is Rudin-Keisler reducible to $\mathcal{G}_B$.

Let $\sse^{\Fin^B}$ denote the following partial order on $\mathcal{P}(B)$:
For $X,Y\sse B$, $Y\sse^{\Fin^B} X$ if and only if $Y\setminus X\in\Fin^B$.
By a routine induction on the rank of $B$, one shows that 
$\mathcal{P}(B)/\Fin^B$ 
is forcing equivalent to $(\mathcal{P}(B),\sse^{\Fin^B})$.
In the next section, we construct topological Ramsey spaces $\mathcal{E}_B$ which, when considered as partially ordered by $\sse^{\Fin^B}$, form dense subsets of  $(\mathcal{P}(B),\sse^{\Fin^B})$.


\section{Infinite dimensional Ellentuck Spaces}\label{sec.defR}

This section contains the construction of the spaces 
$\mathcal{E}_B$, $B$ a uniform barrier of infinite rank, and Theorem \ref{thm.E_BtRs}
proving that each one is a topological Ramsey space.
These spaces extend the hierarchy of the finite dimensional Ellentuck  spaces $\mathcal{E}_k$, $1\le k<\om$, in \cite{DobrinenJSL15},
 $\mathcal{E}_1$ denoting the Ellentuck space here.
Each new space  $\mathcal{E}_B$ is modeled on  the structure of some uniform barrier $B$ on $\om$ of infinite rank and is 
 constructed so that $(\mathcal{E}_B,\sse^{\fin^B})$  is forcing equivalent to $\mathcal{P}(B)/\Fin^B$.
The spaces form a hierarchy in that 
whenever $C$ is a projection of $B$, then $\mathcal{E}_C$ is isomorphic to a projection of $\mathcal{E}_B$ (Fact \ref{fact.projhierarchy}).
Similarly to the way that uniform barriers of rank $\al$ are constructed
by recursion using uniform barriers of rank less than $\al$,
so too the space $\mathcal{E}_B$ is  formed using spaces $\mathcal{E}_C$ for $C$ with rank less than the rank of $B$.
Copies  of  smaller dimensional Ellentuck spaces are seen inside $\mathcal{E}_B$  both as projections of $\mathcal{E}_B$ and as upward images above some fixed root via the recursive construction (Fact \ref{fact.important}).
We point out 
that for any two spaces $\mathcal{E}_B$ and $\mathcal{E}_C$, one is embeddable into the other if we allow for restricting  below some  members of the spaces (Fact \ref{fact.EBCemb}).

Fix the following notation.

\begin{notation}\label{notn.hatB}
Let $\om^{\not\,\downarrow<\om}$  denote the collection of finite non-decreasing sequences of members of $\om$.
For $s\in \om^{\not\,\downarrow<\om}$, let $\lh(s)$ denote the {\em length of $s$}, that is, the cardinality of the domain of the sequence $s$.
Let  $<_{\mathrm{lex}}$ denote the lexicographic ordering on  $\om^{\not\,\downarrow<\om}$, where we also consider any proper initial segment of a sequence to be lexicographically below that sequence.
Given $s,t\in \om^{\not\,\downarrow<\om}$,
define $s\prec t$ if and only if either
\begin{enumerate}
\item
$\max(s)<\max(t)$, or
\item
$\max(s)=\max(t)$ and
$s<_{\mathrm{lex}}t$.
\end{enumerate}
Thus, $\prec$   well-orders
 $\om^{\not\,\downarrow<\om}$ in order-type $\om$, with  the empty sequence $()$ as the $\prec$-minimum.

Define the function $\sigma:[\om]^{<\om}\ra\om^{\not\,\downarrow<\om}$ as follows:
Let $\sigma(\emptyset)=()$, the empty sequence; and
 for $\{a_0,\dots,a_n\}\in[\om]^{<\om}\setminus \{\emptyset\}$, let
\begin{equation}
\sigma(a)=(a_0,a_1-1,a_2-2,\dots,a_n-n).
\end{equation}
Thus, $\sigma$ is a bijection between the collection of all finite subsets of $\om$ and the collection of all finite non-decreasing sequences of members of $\om$.

For  $B\sse[\om]^{<\om}$,  let $\hat{B}$ denote $\{a\in[\om]^{<\om}:\exists b\in B\, (a\tle b)\}$.
For sequences $s,t\in\om^{\not\,\downarrow<\om}$, we shall also use the notation $s\tle t$ to denote that $s=t\re m$
for some $m\le \lh(t)$.
Write $s\tl t$ if and only if $s \tle t$ and $s\ne t$.
For $S\sse \om^{\not\,\downarrow<\om}$, 
let $\hat{S}$ denote $\{s\in \om^{\not\,\downarrow<\om}:\exists t\in S\, (s\tle t)\}$.
\end{notation}

Construct the related spaces of finite sequences as follows.

\begin{defn}[The Sequence Spaces $(S_{\hat{B}},\prec)$ and $(S_B,\prec)$]\label{defn.prec}
Let $B$ be a uniform barrier on $\om$.
Let
\begin{equation}
S_{\hat{B}}=\{\sigma(a):a\in\hat{B}\}\mathrm{\ \ and \ \ }
S_B=\{\sigma(b):b\in B\}.
\end{equation}
\end{defn}

Thus,
$S_{\hat{B}}$ and $S_B$ are   collections of  non-decreasing finite sequences of members of $\om$ which retain the structure of $\hat{B}$ and $B$, respectively.
Note that  $\prec$  well-orders  $S_{\hat{B}}$  and $S_B$
in order-type $\om$.
We now define the maximum member $\bW_B$ of the space $\mathcal{E}_B$.

\begin{defn}[The top member $\bW_B$ of $\mathcal{E}_B$]\label{defn.W_B}
Let  $B$  be a uniform barrier on $\om$.
Since $(S_{\hat{B}},\prec)$  has
 order-type $\om$, 
let $\nu_B:(S_{\hat{B}},\prec)\ra (\om,<)$ denote this order isomorphism.
For $s\in S_B$, 
let $\lh(s)$ denote the length of $s$ and 
let
\begin{equation}
\bW_B(s)=\{\nu_B(s\re m): 1\le  m\le  \lh(s)\},
\end{equation}
which  is a member of $[\om]^{\lh(s)}$.
Define
\begin{equation}
\bW_B=\{\bW_B(s): s\in S_B\}.
\end{equation}

Define the map $\psi_B:S_B\ra \mathbb{W}_B$ by
$\psi(s)=\mathbb{W}_B(s)$, for $s\in S_B$.
Let $\rho_B$ denote $\psi_B\circ\sigma$.
\end{defn}

\begin{rem}\label{rem.precorderW}
Note that
$\bW_B$ is a subset of $[\om]^{<\om}$ which is isomorphic to $B$ via the map $\rho_B$.
That is, $\rho_B$ preserves the tree structure of $B$, in particular, lengths of nodes and lexicographic order.
Further, extend $\psi_B$ to a map from $S_{\hat{B}}$ to $\widehat{\bW}_B$ as follows:  For $s\in S_B$ and $1\le m\le \lh(s)$,
let $\psi_B(s\re m)=\{\nu_B(s\re i):1\le i\le m\}$.
Then  $\prec$ on $S_{\hat{B}}$ induces a well-order on  $\widehat{\bW}_B$; so we shall abuse notation and consider 
 $\widehat{\bW}_B$ as also being well-ordered by $\prec$.
\end{rem}

Next, we define the domain spaces $D_B$.
These are the collections of sets $S\sse S_B$ 
which we allow to serve as templates, or domains,  for members of $\mathcal{E}_B$,
 in the sense that $\mathcal{E}_B$ will  be defined to be $\{\rho_B(S):S\in D_B\}$.
Since the definition is by recursion on the rank of $B$, we begin with $B=[\om]^1$,  presenting an equivalent definition of the Ellentuck space $\mathcal{E}$, which,   in the present notation,   is
denoted as  $\mathcal{E}_{[\om]^1}$.
The restriction map $r^B_k(S)$ is defined to give the first $k$ members of  $S$, and the full approximation map $\fa^B_k(S)$ is designed to ensure that the structure of previously formed spaces is preserved in the recursive construction.

\begin{defn}[The Domain Space $D_B$,   $r^B_n$, and $\mathfrak{a}^B_n$]\label{defn.domainspace}
For each uniform barrier $B$, we define the {\em domain space} $D_B$, the {\em restriction map} $r_k^B$ and  the {\em full approximation map} $\mathfrak{a}^B_k$.
The construction is by recursion on the rank of $B$.

Starting with $B=[\om]^1$, let
\begin{equation}
D_{[\om]^1}=\{S\sse S_{[\om]^1}:|S|=\om\}.
\end{equation}
For each $S\in D_{[\om]^1}$ and $k<\om$, 
the $k$-th restriction and the $k$-th full approximation  of $S$ are the same:
Letting $\{s_i:i<\om\}$ enumerate $S$ according to  the $\prec$-well-ordering of $S$, define
\begin{equation}
r^{[\om]^1}_k(S)=\mathfrak{a}^{[\om]^1}_k(S)=\{s_i:i<k\},
\end{equation}
the
 $k$ $\prec$-least members of $S$.

Suppose that for all uniform barriers $C$ of rank less than $\al$,  $D_C$ has been defined, and also $r_k^C(S)$ and $\mathfrak{a}_k^C(S)$ have been defined, for all $S\in D_C$ and $k<\om$.
Given a uniform barrier $B$ of rank $\al$, recalling Notation \ref{notn.B_n}, for each $n<\om$,
$B_{\{n\}}$ is a uniform barrier of some rank  less than $\al$ on $[n+1,\om)$.
Let $C^n$ denote the uniform barrier of rank less than $\al$ on $\om$ obtained by shifting the members of $B_{\{n\}}$ down by $n+1$.
That is,  
$C^n$ is the collection of all sets $\{a_0-(n+1),\dots, a_l-(n+1)\}$, where  $\{a_0,\dots,a_l\}\in B_{\{n\}}$.
Since $\rank(C^n)<\al$,   $D_{C^n}$ is already defined, along with the restriction and full approximation maps.

Given $s=(s_0,\dots,s_l)$, let $s+(n+1)$ denote $(s_0+n+1,\dots,s_l+n+1)$.
Let $\upsilon_n:S_{C^n}\ra S_{B_n}$ be the map 
$\upsilon_n(s)=(n)^{\frown} (s+(n+1))$.
Define
\begin{equation}
 D_{B_n}=\{\upsilon_n(S):S\in D_{C^n}\}.
\end{equation}
For $T\in D_{B_n}$ and $k<\om$,
define 
\begin{equation}
\fa^{B_n}_k(T)=
\upsilon_n(\fa^{C^n}_k(\upsilon_n^{-1}(T))).
\end{equation}
Let $\widehat{\fa^{B_n}_k(T)}$ denote $\{s\re m:s\in \fa^{B_n}_k(T)$ and $1\le m\le \lh(s)\}$,
the closure of $\fa^{B_n}_k(T)$
under initial segments.
Given two subsets $T,T'\sse S_{\hat{B}_n}$,
write $T\prec T'$ if and only if for each $s\in T$ and each $s'\in T'$,
$s\prec s'$.
For $S\sse S_B$ and $n<\om$, let $S_n$ denote $\{s\in S:s\tre (n)\}$.
\vskip.1in

For the sake of efficiency we shall  define $S_B$ and $\fa^B_k$  simultaneously, (though technically the definition of $S_B$ does not depend on the definition of $\fa^B_k$).
A subset 
$S\sse S_B$ is defined to be in $D_B$ if and only if
$\hat{S}\cap \om^1$ is infinite,
and  enumerating $\hat{S}\cap \om^1$ as $\{(n_0),(n_1),\dots\}$,
\begin{enumerate}
\item
For each $i<\om$,
$S_{n_i}:=\{s\in S:s\tre (n_i)\}\in D_{B_{n_i}}$; and
\item
Defining $\fa^B_0(S)=\emptyset$ and
$\fa^B_1(S)=\fa^{B_{n_0}}(S_{n_0})$,
given $k\ge 1$ and $\fa^B_k(S)$, we require that 
\begin{enumerate}
\item[(a)]
$\fa^B_k(S)\prec
\widehat{\fa^{B_{n_0}}_{k+1}(S_{n_0})}\setminus \widehat{\fa^B_k(S)}$, and
\item[(b)]
 for each  $1\le i\le k$,
$\fa^{B_{n_{i-1}}}_{k-i}(S_{n_{i-1}})\prec
\widehat{\fa^{B_{n_i}}_{k+1-i}(S_{n_i})}\setminus
\widehat{\fa^{B_{n_i}}_{k-i}(S_{n_i})}$.
\end{enumerate}
Define 
$\fa^B_{k+1}(S)=
\fa^B_{k}(S)\cup \fa^{B_{n_0}}_{k+1}(S_{n_0})
\cup \dots\cup \fa^{B_{n_k}}_1(S_{n_k})$.
\end{enumerate}
Finally,
the restriction map $r^B$ is simply defined by letting $r^B_k(S)$ be the set of the $k$ $\prec$-least members of $S$.
\end{defn}

\begin{rem}
First,  notice that (2)  is equivalent to requiring that
\begin{align}
&\mathfrak{a}^{B_{n_0}}_2(S_{n_0})
\prec 
\widehat{\mathfrak{a}^{B_{n_1}}_1(S_{n_1})}\cr
\prec\  &   \widehat{ \mathfrak{a}^{B_{n_0}}_3(S_{n_0})}\setminus \widehat{\mathfrak{a}^{B_{n_0}}_2(S_{n_0})}
\prec 
\widehat{\mathfrak{a}^{B_{n_1}}_2(S_{n_1})}\setminus \widehat{\mathfrak{a}^{B_{n_1}}_1(S_{n_1})}
\prec 
\widehat{\mathfrak{a}^{B_{n_2}}_1(S_{n_2})}\cr
\prec \ 
&\widehat{\mathfrak{a}^{B_{n_0}}_4(S_{n_0})}\setminus \widehat{\mathfrak{a}^{B_{n_0}}_3(S_{n_0})}
\prec \widehat{\mathfrak{a}^{B_{n_1}}_3(S_{n_1})}\setminus \widehat{\mathfrak{a}^{B_{n_1}}_2(S_{n_1})}\prec\dots
\end{align}
Second, the $k+1$-st full approximation
$\fa_{k+1}^B(S)$
is equal to  $\fa_{k+1}^{B_{n_0}}(S_{n_0})\cup \fa_k^{B_{n_1}}(S_{n_1})\cup\dots\cup \fa_1^{B_{n_k}}(S_{n_k})$.
Third, the last member of a full approximation is free to range over a tail of $S$.
That is, given $S\in D_B$ and $k\ge 1$,  letting $n$ be such that $r_{n+1}^B(S)=\fa_k^B(S)$,
the set $\{s\in S: r_n^B(S)\cup\{s\}=\fa_k^B(S')$ for some $S'\in D_B\}$ is exactly the set $\bigcup\{S_m:m\ge\max(r^B_n(S))\}$, which is again a member of $D_B$.
\end{rem}

Now we  have the correct machinery to define the  spaces which extend the Ellentuck space up to  all countably infinite ranks.

\begin{defn}[The Infinite Dimensional Ellentuck Spaces $(\mathcal{E}_B,\le,r)$]\label{def.E_B}
For a uniform barrier $B$, let
\begin{equation}\label{eq.defE_B}
\mathcal{E}_B=\{\psi_B(S):S\in D_B\}.
\end{equation}
For $X,Y\in\mathcal{E}_B$,
$Y\le X$ if and only if $Y\sse X$.
Given $X\in\mathcal{E}_B$ and  $k<\om$, letting $S$ be the member of $D_B$ such that $X=\psi_B(S)$, 
define the $k$-th restriction of $X$ to be
$r_k^B(X)=\psi_B(r^B_k(S))$ and  the $k$-th full approximation of $X$ to be $\mathfrak{a}_k^B(X)=\psi_B(\mathfrak{a}_k^B(S))$.
\end{defn}

\begin{fact}\label{fact.projhierarchy}
If $B$ and $C$ are uniform barriers such that $C\tl B$, then the projection map from $B$ to $C$ induces a projection from $\mathcal{E}_B$ to $\mathcal{E}_C$.
\end{fact}

This follows from  noting that the map $\rho_C\circ\pi_{C,B}\circ\rho_B^{-1}:\mathcal{E}_B\ra\mathcal{E}_C$ 
preserves that
  $\prec$, $\tle$, and $<_{\mathrm{lex}}$ structures on the stems which are not projected,
where $\pi_{C,B}$ denotes the projection map from $B$ to $C$.

\begin{example}[The Ellentuck space based on the Shreier barrier]\label{example.E_shreier}
Let $\mathcal{S}$ denote the {\em Shreier barrier}, the set $\{b\in[\om]^{<\om}: \lh(b)=\max(b)+1\}$.
Then the first few members in the $\prec$-well-ordering of $S_{\mathcal{S}}$  are 
$()\prec (0)\prec (1)\prec (1,1)\prec (1,2)\prec (2)\prec (2,2)\prec(2,2,2)\prec (1,3)\prec (2,2,3)\prec (2,3)\prec(2,3,3)\prec (3)\prec (3,3)\prec (3,3,3)\prec (3,3,3,3)\prec(1,4)\prec\dots$.
The tree structure of sequence space $S_{\mathcal{S}}$  is the following.

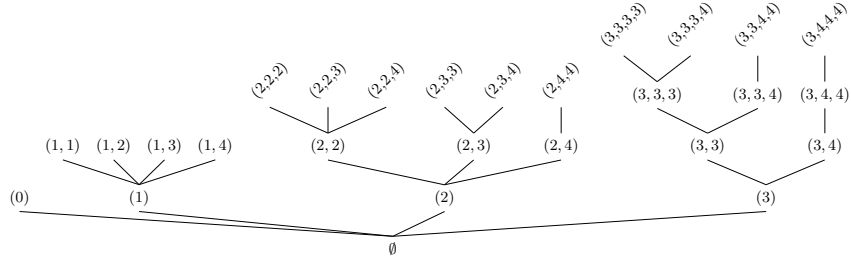
\begin{figure}[\h]
\centering
{\footnotesize
\begin{tikzpicture}[scale=.65,grow'=up, level distance=30pt,sibling distance=.1cm]
\tikzset{grow'=up}
\Tree [.$\emptyset$ [.$(0)$ ][.$(1)$ [.$(1,1)$  ] [.$(1,2)$ ] [.$(1,3)$ ]  [.$(1,4)$ ] ][.$(2)$ [.$(2,2)$ [.$\rotatebox{45}{(2,2,2)}$ ] [.$\rotatebox{45}{(2,2,3)}$ ] [.$\rotatebox{45}{(2,2,4)}$ ]] [.$(2,3)$ [.$\rotatebox{45}{(2,3,3)}$ ][.$\rotatebox{45}{(2,3,4)}$ ]  ] [.$(2,4)$ [.$\rotatebox{45}{(2,4,4)}$ ]  ] ][.$(3)$  [.$(3,3)$ [.$(3,3,3)$ [.$\rotatebox{45}{(3,3,3,3)}$ ] [.$\rotatebox{45}{(3,3,3,4)}$ ]] [.$(3,3,4)$ [.$\rotatebox{45}{(3,3,4,4)}$ ] ] ]
 [.$(3,4)$ [.$(3,4,4)$ [.$\rotatebox{45}{(3,4,4,4)}$ ] ]  ] ]]
\end{tikzpicture}
}
\caption{$S_{\mathcal{S}}$}
\end{figure}

The maximum member of the space $\mathcal{E}_{\mathcal{S}}$ is $\bW_{\mathcal{S}}$.

\begin{figure}[\h]
\centering
{\footnotesize
\begin{tikzpicture}[scale=.55,grow'=up, level distance=30pt,sibling distance=.1cm]
\tikzset{grow'=up}
\Tree [.$\emptyset$ [.$\{0\}$ ][.$\{1\}$ [.$\{1,2\}$  ] [.$\{1,3\}$ ] [.$\{1,7\}$ ]  [.$\{1,15\}$ ] ][.$\{4\}$ [.$\{4,5\}$ [.$\rotatebox{45}{\{4,5,6\}}$ ] [.$\rotatebox{45}{\{4,5,8\}}$ ] [.$\rotatebox{45}{\{4,5,16\}}$ ]] [.$\{4,9\}$ [.$\rotatebox{45}{\{4,9,10\}}$ ][.$\rotatebox{45}{\{4,9,17\}}$ ]  ] [.$\{4,18\}$ [.$\rotatebox{45}{\{4,18,19\}}$ ]  ] ][.$\{11\}$  [.$\{11,12\}$ [.$\{11,12,13\}$ [.$\rotatebox{45}{\{11,12,13,14\}}$ ] [.$\rotatebox{45}{\{11,12,13,20\}}$ ]] [.$\{11,12,21\}$ [.$\rotatebox{45}{\{11,12,21,22\}}$ ] ] ]
 [.$\{11,23\}$ [.$\{11,23,24\}$ [.$\rotatebox{45}{\{11,23,24,25\}}$ ] ]  ] ]]
\end{tikzpicture}
}
\caption{$\mathbb{W}_{\mathcal{S}}$}
\end{figure}
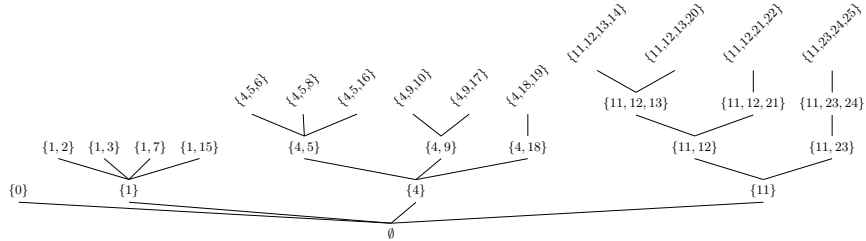

The $k$-th restriction of $\bW_{\mathcal{S}}$ consists of the $k$ $\prec$-least members of  $\bW_{\mathcal{S}}$; for instance, $r^{\mathcal{S}}_1(\bW_{\mathcal{S}})=\{\{0\}\}$,
$r^{\mathcal{S}}_3(\bW_{\mathcal{S}})=\{\{0\},\{1,2\},\{1,3\}\}$, and 
$r^{\mathcal{S}}_5(\bW_{\mathcal{S}})=\{\{0\},\{1,2\},\{1,3\},\{1,7\},\{4,5,8\}\}$.

The first full approximation of $\bW_{\mathcal{S}}$ is $\fa^{\mathcal{S}}_1(\bW_{\mathcal{S}})=\{\{0\}\}
=r^{\mathcal{S}}_1(\bW_{\mathcal{S}})$.
The second full approximation is
$\fa^{\mathcal{S}}_2(\bW_{\mathcal{S}})=\{\{0\},\{1,2\}\}=
r^{\mathcal{S}}_2(\bW_{\mathcal{S}})$.
$\fa^{\mathcal{S}}_3(\bW_{\mathcal{S}})=
\{\{0\},\{1,2\}, \{1,3\},\{4,5,6\}\}=
r^{\mathcal{S}}_4(\bW_{\mathcal{S}})$
and 
$\fa^{\mathcal{S}}_4(\bW_{\mathcal{S}})$ equals the following:
$$
\{\{0\},\{1,2\}, \{1,3\},\{4,5,6\},\{1,7\},\{4,5,8\}, \{4,9,10\},\{11,12,13,14\}\},$$
which is equal to $r^{\mathcal{S}}_8(\bW_{\mathcal{S}})$.
We point out that for each $k<\om$, $\min(\bW_{\mathcal{S}}(k)\setminus r^{\mathcal{S}}_k(\bW_{\mathcal{S}}))>
\max(r^{\mathcal{S}}_k(\bW_{\mathcal{S}}))$.
\end{example}

\begin{notation}\label{notn.EB_a}
$\mathcal{AE}^B_k$ denotes the set $\{r^B_k(X):X\in\mathcal{E}_B\}$, and $\mathcal{AE}^B$ denotes $\bigcup_{k<\om}\mathcal{AE}^B_k$.
For $X\in\mathcal{E}_B$ and $w,w'\in X$, note that $\max(w)<\max(w')$ if and only if $\psi_B^{-1}(w)\prec  \psi_B^{-1}(w')$.
Thus, we will often also use $\prec$ to denote the well-ordering of members of $X$.
For $n<\om$, $X(n)$ denotes the $w\in X$  such that $r_n(X)\cup \{w\}=r_{n+1}(X)$.
Note that  $X(n)$ 
 has the $n$-th smallest value $\max(w)$ in $X$, and is the $n$-th least member in the $\prec$ ordering of $X$.

Given $a\in\hat{B}\setminus B$,
let $B_a$ denote $\{b\in B:b\tr a\}$ and
$w_a=\rho_B(a)$.
For $X\in\mathcal{E}_B$
and
any $w_*\in\hat{X}\setminus X$, 
 let $X_{w_*}$ denote $\{w\in X:w\tr w_*\}$.
Define
 $\mathcal{E}_{B_a}=\{X_{w_a}:X\in\mathcal{E}_B\}$.
Further, let $n=\max(a)$, and let $\delta_n:[n+1,\om)\ra\om$ be the `shift down by $n+1$' map given by $\delta_n(k)=k-n-1$.
Letting  $B^a$ denote $\{\delta_n(b\setminus a):b\in B_a\}$, we see that $B^a$ 
 is a uniform barrier on $\om$.
\end{notation}

\begin{rem}
The idea behind the full approximation map is  to ensure that our inductive construction yields the property that given any $a\in\hat{B}\setminus B$,
the collection $\{X_{w_a}:X\in\mathcal{E}_B\}$ is isomorphic to
the space $\mathcal{E}_{B^a}$.
This next fact is one of the key properties of these spaces, justifying the use of  terminology of infinite dimensional {\em Ellentuck space}.
\end{rem}

\begin{fact}\label{fact.important}
For $a\in\hat{B}\setminus B$,
the space $\mathcal{E}_{B_a}$ is isomorphic to   $\mathcal{E}_{B^a}$, via the map $X_{w_a} \mapsto \{  \rho_{B^a}\circ\delta_n(  \rho_B^{-1}(w)    \setminus a):w\in X_{w_a}\}$ for $X_{w_a}\in \mathcal{E}_{B_a}$,
where $n=\max(a)$.
\end{fact}

In particular, for each $b\in B$, letting $a=b\setminus\{\max(b)\}$, $\mathcal{E}_{B_a}$ is isomorphic to {\em the} Ellentuck space.
The fact is proved by fixing $B$ and doing a straightforward induction on the rank of $B_a$, using the recursive definition of the domain spaces $D_{B_a}$.

The next fact shows that the set of all 1-extensions of a given $u\in\mathcal{AE}^B$ is also isomorphic to $\mathcal{E}_C$ for some uniform front $C$.

\begin{fact}\label{fact.inductive}
Let   $u\le_{\fin} X\in\mathcal{E}_B$.
There is a unique $w_u\in\hat{X}\setminus X$ such that $r^B_{|u|+1}[u,X]$ is equal to the set
\begin{equation}
\{u\cup \{w\}\in X:  w\vartriangleright w_u\mathrm{\ and\ } \min(w\setminus w_u)>\max(u)\}.
\end{equation}
Let $X_u$ denote $\{w\in X:w\tr w_u$ and $\min(w\setminus w_u)>\max(u)\}$, and let $a_u$ denote $\rho_B^{-1}(w_u)$.
Then $X_u=
\{v(|u|):v\in r^B_{|u|+1}[u,X]\}$ and  is a member of $\mathcal{E}_{B_{a_u}}$.
In fact, $\{X_u:X\in\mathcal{E}_B$ and $u\le_{\fin}X\}$ 
 is  a dense subset of  
 $\mathcal{E}_{B_{a_u}}$.
\end{fact}

\begin{proof}
The proof follows by induction on $\rank(B)$ using the recursive definition of $r_{|u|+1}^B[u,X]$.
The statement we are actually proving is that there is a unique $w_u$ such that $X_u$ is the collection of last members of a full approximation in $\mathcal{E}_{B_{a_u}}$ which extend $u\cap X_{w_u}$ by one element.
For $\rank(B)<\om$, this fact was  shown in \cite{DobrinenJSL15}.
Let $B$ have infinite rank and suppose the fact holds for all $C$ with smaller rank.
If $r_{|u|+1}^B(X)=\fa_k^B(X)$ for some $k$, then $X_u$ is exactly $\{w\in X: \min(w)>\max(u)\}$, which is a member of $\mathcal{E}_B$.
In this case, $w_u=\emptyset$ and $\{X_u:X\in\mathcal{E}_B$ and $u\le_{\fin} X$ is a dense subset of $\mathcal{E}_B$.
If $r_{|u|+1}^B(X)$ is not a full approximation, 
then 
by the recursive construction of $D_B$, there is some 
 $a\in\hat{B}\setminus B$ such that the 1-extensions of $u$ in $\mathcal{E}_B$ are according to 1-approximations in the structure of $D_{B_a}$.
Thus, by the induction hypothesis, there is a unique $w_u\tre\rho_B(a)$ such that  the collection of 1-extensions of $u$ into $X$ in $\mathcal{E}_B$ is exactly
$X_u$, which is a member of $\mathcal{E}_{B_a}$.
Then $a_u=a$.
\end{proof}

We  point out that for $u\le_{\fin}X\in\mathcal{E}_B$, letting 
$S=\psi_B^{-1}(X)$ and
$S_u=\psi_B^{-1}(X_u)$, then 
 $S_u$ is equal to $\{s\in S: s\tr \psi_B^{-1}(w_u)$ and $\min(\psi_B(s)\setminus w_u)>\max(u)\}$ and is a member of  $D_{B_{a_u}}$.

In the next Lemma, 
use the following notation:
For $S\in D_B$, let $S_n$ denote $\{s\in S:S\tre (n)\}$; and  
for 
 $X\in\mathcal{E}_B$,  let $X_n$ denote $\{w\in X: w\trianglerighteq \rho_B(\{n\})\}$, and note that 
this is  equal to  $\{w\in X: w\trianglerighteq \psi_B((n))\}$.
Note that for each $n$,
either $X_n\in\mathcal{E}_{B_n}$ or else $X_n=\emptyset$.

The following two Fusion Lemmas will be used  in numerous diagonal constructions.

\begin{lem}[Fusion]\label{lem.fusion1}
Suppose that $S\in D_B$ and for 
 infinitely many  $n<\om$, $S'_n$ is  a member of $D_{B_n}$ such that $S'_n\sse S_n$.
Then there is an $S''\in D_B$ such that for each $n<\om$ for which $S''_n\ne\emptyset$, we have  $S''_n\sse S_n'$.

Hence,
if $X\in\mathcal{E}_B$ and for infinitely many $n<\om$, 
$Y_n\in\mathcal{E}_{B_n}$ with $Y_n\sse X_n$,
then there is a $Z\le X$ in $\mathcal{E}_B$ such that 
for all $\rho_B(\{n\})\in \hat{Z}$,
$Z_n\sse Y_n$.
\end{lem}

\begin{proof}
Let $S\in D_B$ and $N\in[\om]^{\om}$,
and suppose that for each $n\in N$,
  $S'_n$ is a member of $D_{B_n}$ such that $S'_n\sse S_n$.
The following is a mechanism for constructing $S''\in D_B$ by constructing its full approximations so that for each $(n)\in S''$, $S''_n\sse S'_n$.
Let $n_0$ be the least member of $N$ and define  $\fa^B_2(S'')$ to be $\fa^{B_{n_0}}_2(S'_{n_0})$.
Let $n_1$ be least in $N$ such that 
$n_1>\max \fa^B_2(S'')$.
Then take  $\fa^B_2(S'')$ to be $\fa^{B_{n_0}}_2(S'_{n_0})\cup \fa_1^{B_{n_1}}(S'_{n_1})$. 
Now let $m(2,0)=\max\fa^B_2(S'')$.
Extend $\fa_2^{B_{n_0}}(S''_{n_0})$ to any $\fa_3^{B_{n_0}}(S''_{n_0})\sse S'_{n_0}$ such that 
the minimum of 
$\widehat{\fa_3^{B_{n_0}}(S''_{n_0})}  \setminus \widehat{\fa_3^{B_{n_0}}(S''_{n_0})}$ is greater than $m(2,0)$.
Next, let $m(2,1)$ be the maximum of $\fa_3^{B_{n_0}}(S''_{n_0})$ and extend $\fa_1^{B_{n_1}}(S''_{n_1})$ to any $\fa_2^{B_{n_1}}(S''_{n_1})\sse S'_{n_1}$ such that 
the minimum of 
$\widehat{\fa_2^{B_{n_1}}(S''_{n_1})}  \setminus \widehat{\fa_1^{B_{n_1}}(S''_{n_1})}$ is greater than $m(2,1)$.
Then let $n_2$ be the least member of $N$ greater than  maximum of $\fa_2^{B_{n_1}}(S''_{n_1})$.
Let $\fa^B_3(S'')=
\fa^{B_{n_0}}_3(S''_{n_0})\cup \fa^{B_{n_1}}_2(S''_{n_1})\cup \fa_1^{B_{n_2}}(S'_{n_2})$.
Continuing in this manner, one constructs $S'' =\bigcup_{i\ge 2} \fa^B_i(S'')$ in $D_B$ and an infinite subset $\{n_i:i<\om\}\sse N$ such that for each $i<\om$,
$S''_{n_i}\sse S'_{n_i}$, and for each $n$ not equal to some $n_i$, $S''_n=\emptyset$.
\end{proof}

\begin{lem}[Fusion]\label{lem.fusion2}
Let $u\le_{\fin} X\in\mathcal{E}_B$.
Let $N$ be an infinite set of integers greater than $\max(a_u)$,
and for each $n\in N$, let
$Y_n\in\mathcal{E}_{B_{a_n}}$
with
$Y_n\sse \{w\in X:w\tre \rho_B(a_n)\}$, where
 $a_n$ denotes $a_u\cup\{n\}$.
Then there is a $Z\in [u,X]$ such that 
 for each $n>\max(a_u)$, $Z_n\sse Y_n$.

In particular,
given $u\le_{\fin} X\in\mathcal{E}_B$ and $Y_u\in\mathcal{E}_{B_{a_u}}$ such that $Y_u\sse X_u$,
then there is a $Z\in [u,X]$ such that $Z_u\sse Y_u$.
\end{lem}

\begin{proof}
Apply Fusion  Lemma \ref{lem.fusion1} in  $\mathcal{E}_{B_{a_u}}$ to
 the collection of $Y_n$,  $n\in N$, to find a $U\in\mathcal{E}_{B_{a_u}}$ such that for each $n\in\om$,
either $U_n=\emptyset$, or else both $n\in N$ and $U_n\sse Y_n$,
where $U_n$ denotes $\{w\in U:w\tre \rho_B(a_n)\}$.
Then  extend $u$ into $X$ recursively as follows.
Let $k=|u|$ and let $z_k=u$.
For $m\ge k$, given $z_m$,
if  $w_{z_m}\tre w_u$, then take $z_{m+1}\in r_{m+1}[z_m,X]$ such that $z_{m+1}(m)\in U$;
otherwise, take any $z_{m+1}\in r_{m+1}[z_m,X]$.
Letting $Z=\bigcup_{m\ge k} z_m$ proves the first part.

The second part follows immediately from the construction method of $Z$ in the first part.
\end{proof}

The next fact shows that any two spaces are comparable, allowing for restrictions below some members.

\begin{fact}\label{fact.EBCemb}
For any uniform barriers $B$ and $C$, there are $X\in\mathcal{E}_B$ and $Y\in\mathcal{E}_C$ such that either the subspace $\mathcal{E}_B|X$ embeds into the subspace $\mathcal{E}_C|Y$, or vice versa.
\end{fact}

The proof is by a straightforward induction on the maximum of the ranks of $B$ and $C$ using the Fusion Lemma \ref{lem.fusion1}.
This base case comes from \cite{DobrinenJSL15}, as this fact is proved there for all pairs  of uniform barriers of finite rank.

This section concludes by proving the main theorem of this section, Theorem \ref{thm.E_BtRs}:
 Each  space $(\mathcal{E}_B,\le,r)$  is a topological Ramsey space; hence, every  subset of $\mathcal{E}_B$  with the property of Baire is Ramsey.
Since   $\mathcal{E}_B$  is a closed subspace of $(\mathcal{AE}^B)^{\mathbb{N}}$,
it suffices, by  the Abstract Ellentuck Theorem  \ref{thm.AET}, 
to show that  $(\mathcal{E}_B,\le,r)$
 satisfies the axioms \bf A.1 \rm -  \bf A.4\rm.
Leaving the routine checking of 
 axioms  \bf A.1 \rm  and  \bf A.2 \rm  to the reader,
we first show that \bf A.3 \rm holds for $\mathcal{E}_B$ 
for all  uniform barriers $B$.
Then  
we  show by induction on rank of $B$ that \bf A.4 \rm holds for $\mathcal{E}_B$, assuming that $\mathcal{E}_C$ for $C$ of smaller rank than $B$ have already been proved to be topological Ramsey spaces.

\begin{fact}\label{fact.true}
Given $u\in\mathcal{AE}^B$, if $u\sse X\in\mathcal{E}_B$, then $r_{|u|+1}[u,X]\ne\emptyset$.
\end{fact}

\begin{lem}\label{lem.A.3}
For each $B$, the space $(\mathcal{E}_B,\le,r)$ satisfies Axiom \bf A.3\rm.
\end{lem}

\begin{proof}
Given $u\in\mathcal{AE}^B$ and $X\in\mathcal{E}_B$ such that $\depth_X(u)=d<\infty$, and letting $Y\in[d,X]$, construct  $Z\in[u,Y]$ as follows.
Let $z_m=u$, where $m=|u|$.
Note that $z_m\sse X$.
Given $n\ge m$ and $z_n\sqsupseteq u$ such that $z_n\sse X$,
by Fact \ref{fact.true}, $r_{n+1}[z_n,X]\ne\emptyset$, so  we may take $z_{n+1}\in r_{n+1}[z_n,X]$.
Then $z_{n+1}\sse Z$, so the induction process continues.
Letting $Z=\bigcup_{n\ge m} z_n$, we see that $Z\in [u,X]$;
hence  \bf A.3 \rm (1) holds.

For \bf A.3 \rm(2), 
suppose $Y\le X$ and $[u,Y]\ne\emptyset$.
Let $d=\depth_X(u)$ and  let  $z_d=r_d(X)$.
For  $n\ge d$ with $z_n$ already chosen,
if $w_{z_n}\in\hat{u}$, then choose $z_{n+1}\in r_{n+1}[z_n,X]$
such that 
$z(n)\in Y$.
Otherwise, simply choose any $z_{n+1}\in r_{n+1}[z_n,X]$.
Let $Z=\bigcup_{n\ge d} z_n$.
Then $Z\in[d,X]$.
Now if $Z'\in[u,Z]$, then for each $n<\om$, there is a $w\in\hat{u}$ such that $Z'(n)\vartriangleright w$, which implies that $Z'(n)\in Y$.  Hence, $Z'\in[u,Y]$, so
 $[u,Z]\sse[u,Y]$.
\end{proof}


\begin{thm}\label{thm.E_BtRs}
For each uniform barrier $B$ on $\om$,
$(\mathcal{E}_B,\le, r)$ is a topological Ramsey space.
\end{thm}

\begin{proof}
The proof is by induction on $\rank(B)$.
If $\rank(B)=1$, then $\mathcal{E}_B$ is the Ellentuck space; 
for  $2\le k<\om$,  $\mathcal{E}_{[\om]^k}$ is the space $\mathcal{E}_k$ which is a topological Ramsey space by Theorem  21 in \cite{DobrinenJSL15}.
Now suppose $\rank(B)=\al\ge \om$, and for all uniform barriers $C$ with $\rank(C)<\al$, $\mathcal{E}_C$ is a topological Ramsey space.
It suffices to prove the Pigeonhole Principle \bf A.4 \rm for $\mathcal{E}_B$.

Let $X\in\mathcal{E}_B$,  $u=r_k(X)$, and $\mathcal{O}\sse\mathcal{AE}^B_{k+1}$, and 
let $X_u$ denote $\{v(k):v\in r_{k+1}[u,X]\}$.
Recalling that $a_u$ denotes $\rho^{-1}_B(w_u)$,
we point out that $X_u\in\mathcal{E}_{B_{a_u}}$
and $\{\{w\}:w\in X_u\}\sse\mathcal{AE}_1^{B_{a_u}}$.
If $w_u\ne\emptyset$, then by the induction hypothesis, $\mathcal{E}_{B_{a_u}}$ is a topological Ramsey space.
Define a coloring $c:X_u\ra 2$ by 
$c(w)=0$ if and only if $u\cup\{w\}\in \mathcal{O}$.
By the Abstract Nash-Williams Theorem for $\mathcal{AE}_1^{B_{a_u}}$, there is a 
 $Y_u\in \mathcal{E}_{B_{a_u}}|X_u$ such that 
each $w\in Y_u$ has the same color.
By the Fusion Lemma \ref{lem.fusion2},
there is a $Z\in [u,X]$ such that $Z_u\sse Y_u$.
Thus, $Z$ satisfies \bf A.4\rm.

Suppose now that  $w_u=\emptyset$.
Then $X_u$ equals the set of all $w\in X$ for which $\min(w)>\max(u)$.
Let $N$ denote the set of 
$n >\max(u)$  for which  $X_n\ne\emptyset$.
For each $n\in N$, 
let $X_n$ denote $\{w\in X:w\tr\{n\}\}$ and
let $a_n$ denote $\rho_B^{-1}(\{n\})$.
Since $\rank(B_{a_n})<\al$, the induction hypothesis yields that
 $\mathcal{E}_{B_{a_n}}$ is a topological Ramsey space.
Let $c_n:X_n\ra 2$ be the coloring such that $c_n(w)=0$ if and only if $u\cup w\in \mathcal{O}$, for $w\in X_n$.
Since the set of singletons  $\{\{w\}:w\in X_n\}$ is actually a subset of $\mathcal{AE}^{B_{a_n}}_1$,
 the  Abstract  Nash-Williams Theorem for $\mathcal{E}_{B_{a_n}}$  implies that there is a  $Y_n\in \mathcal{E}_{B_{a_n}}$  such that $Y_n\sse X_n$ and $Y_n$ is homogeneous for the coloring $c_n$.
Take an infinite set $M\sse N$ such that  $c_n\re Y_n$ have the same color for all 
  $n\in M$.
By the Fusion Lemma \ref{lem.fusion2},
there is a $Z\in [u,X]$ such that for each $n>\max(u)$ for which $Z_n\ne\emptyset$, $n$ must be  in $M$ and $Z_n\sse Y_n$.
Thus, \bf A.4 \rm is satisified by $Z$.
\end{proof}


\section{Uniform projections and canonical equivalence relations on 1-extensions}\label{sec.canoneqrelAE1}


In \cite{DobrinenJSL15},
we proved  that, for $2\le k<\om$, each equivalence relation on $\mathcal{AE}_1^{[\om]^k}|X$ is essentially one of $k+1$  canonical types:
 projections to $[\om]^n$ for $n\le k$.
That is, there is some $n\le k$ and some $Y\le X$ such that for all pairs $w=\{i_0,\dots,i_{k-1}\}$ and $w'=\{j_0,\dots,j_{k-1}\}$ of elements in $Y$, $w$ and $w'$ are equivalent if and only if their projections to length $n$,
$\pi_n(w)=\{i_l:l<n\}$ and $\pi_n(w')=\{j_l:l<n\}$, are equal.

For uniform barriers  $B$ of infinite rank, there are no longer only finitely many canonical equivalence relations on $\mathcal{AE}_1^B$;
in fact, there are continuum many.
However, just as the projections  $\pi_n$, $n\le k$, 
on $\mathcal{AE}_1^{[\om]^k}$
gave images which were of the form of the uniform barrier $[\om]^n$,
so too 
the structure of a canonical equivalence relation  on $\mathcal{AE}_1^B$ for an infinite rank uniform barrier $B$
closely resembles the structure of a projection of $B$ to some uniform barrier consisting of initial segments of members of $B$.
We  make this precise below.

The following notation will be used throughout.

\begin{notation}\label{notn.yes}
We shall often use the meet symbol $w \wedge w'$ to denote $w\cap w'$ for $w,w'\in\widehat{\bW}_B$, since the tree structure of $\widehat{\bW}_B$ under initial segments is quite important.
For  $X\in\mathcal{E}_B$,
$\hat{X}$ denotes $\{w'\in [\om]^{<\om}:\exists w\in X\, (w\tre w')\}$.
For $w_*\in\hat{X}\setminus X$, let $X_{w_*}=\{w\in \hat{X}:w\tr w_*\}$.
Note that $X_{w_*}\in\mathcal{E}_{B_{a_*}}$, where $a_*=\rho^{-1}_B(w_*)$.
For $u\in\mathcal{AE}^B$,
let $w_u$ denote the $\vartriangleleft$-maximal member of $\widehat{\mathbb{W}}_B$ such that for all $X\in [u,\mathbb{W}_B]$,
$X(|u|)\vartriangleright w_u$.
Note then that for any $X\ge_{\fin}u$, $w_u$ is equal to the meet  $\bigwedge\{v(|u|):v\in r_{|u|+1}[u,X]\}$.
Let $a_u$ denote $\rho_B^{-1}(w_u)$.
For $u\le_{\fin}X$,
recall that
$X_u=\{w\in X_{w_u}:\min(w\setminus w_u)>\max(u)\}$ is a member of $\mathcal{E}_{B_{a_u}}$, and
 is exactly the set of those $w\in X$ such that $u\cup\{w\}\in r_{|u|+1}[u,X]$, and that $X_u\in\mathcal{E}_{B_{a_u}}$.
\end{notation}

Given a  uniform barrier $B$,  for each $X\in\mathcal{E}_B$ and $w_*\in\hat{X}\setminus X$,
we will define the set of all {\em uniform projections on} 
$X_{w_*}$
by induction on  $\rank(B_{a})$, where $a=\rho^{-1}_B(w_*)$.

\begin{defn}[Uniform Projections $\UP(B,X,w)$]\label{defn.UPBXs}
Let $X\in\mathcal{E}_B$ and $w\in\hat{X}\setminus X$, and let $a$ denote $\rho_B^{-1}(w)$.
If
$\rank(B_{a})=1$, 
define 
$\UP(B,X,w)=\{\{w\}, X_{w}\}$.
Assume now  that for all triples $(B,X,w)$ 
with $\rank(B_{a})<\al$,
 the collection  $\UP(B,X,w)$ has been defined.
Let $(B,X,w)$ be a triple with $\rank(B_{a})=\al$.
A subset $P\sse 
\hat{X}_{w}$ is in $\UP(B,X,w)$ if and only if the following (1) - (4) hold:
\begin{enumerate}
\item 
$\forall p\in P$, 
$p\tre w$;
\item
$P$ is a {\em front on $X_{w}$}; by this we mean that for each $w'\in X_{w}$, there is a $p\in P$ such that $p\tle w'$, and for all $p,p'\in P$, $p\tle p'$ implies $p=p'$.
\end{enumerate}
Enumerating those $n>\max(a)$ such that $a\cup\{n\}\in
\rho_B^{-1}(\hat{X}_{w})$ as $n_0<n_1<\dots$,
and letting $p_i=\rho_B(a\cup\{n_i\})$,
\begin{enumerate}
\item[(3)]
For all $i<\om$,
$P_i:=\{p\in P:  p\tre p_i\}$ is  in $\UP(B_{a\cup\{n_i\}},X,p_i)$; and
\item[(4)]
Either
\begin{enumerate}
\item[(a)]
 there is a $\beta<\al$ such that for all $i<\om$, $\rank(P_i)=\beta$; or
\item[(b)]
there is a strictly increasing sequence $\beta_i$ with $\sup_{i<\om}\beta_i=\al$ such that for each $i<\om$,
$\rank(P_i)=\beta_i$.
\end{enumerate}
\end{enumerate}
\end{defn}

\begin{rem}
 $\rho^{-1}_B(P)$ is usually not a barrier or even a front on some infinite subset of $\om$, but it does have the structure of a uniform barrier.
The above construction recursively defines  the notion of uniformness.
Further, the construction recursively defines the notion of rank of a uniform projection
 in the same way that the recursive definition of uniform barrier defines the notion of rank of a uniform barrier.
\end{rem}

Given  $u\le_{\fin}X\in\mathcal{E}_B$  
let $\UP(B,X,u)$ denote $\{P\in\UP(B,X,w_u):P\sse \hat{X}_u\}$.
This is 
the set of  {\em uniform projections on}  $X_u$.
Given $P\in\UP(B,X,u)$ and $Y\le X$ with $u\le_{\fin} Y$,
we shall write $P|Y$
to denote
$P\cap \hat{Y}_u$.
Note that $P|Y\in\UP(B,Y,u)$.

Given $P,Q\in\UP(B,X,u)$, write $P\tl Q$ if and only if for all pairs $p\in P$ and $q\in Q$,
either $p\tl q$ or else $p$ and $q$ are $\tle$-incomparable.
The next lemma shows that and two members of $\UP(B,X,u)$ are $\tl$-comparable  or equal on some smaller $Y\le X$.

\begin{lem}\label{lem.CsqsubD}
Let $u\le_{\fin}X\in\mathcal{E}_B$, and let $P,Q\in\UP(B,X,u)$.
Then there is a $Y\in[u,X]$ such that $P|Y\tl Q|Y$, $P|Y=Q|Y$, or $P|Y\tr Q|Y$.
\end{lem}

\begin{proof}
The proof is by induction on rank of $B_{a_u}$, over all pairs $u\le_{\fin}X\in\mathcal{E}_B$.
If $\rank(B_{a_u})=1$, then 
$\UP(B,X,u)$ is equal to the set $\{\{w_u\}, X_u\}$, 
these uniform projections having rank $0$ and $1$, respectively.
The lemma immediately follows for this case.

Now suppose that  for all $u\le_{\fin}X$ with $\rank(B_{a_u})<\al$,
there is a $Y\in[u,X]$ such that the pair $P|Y$, $Q|Y$ satisfy the lemma.
Fix  $u\le_{\fin} X\in\mathcal{E}_B$ with  $\rank(B_{a_u})=\al$.
Let $P,Q\in\UP(B,X,u)$, and
let $N$ denote the set of all $n>\max(u)$ for which $w_u\cup\{n\}\in \hat{X}$.
For each $n\in N$, let
 $\gamma_n=\rank(P|X_{w_u\cup\{n\}})$, $\delta_n=\rank(Q|X_{w_u\cup\{n\}})$, and  $a_n=\rho^{-1}_B(w_u\cup\{n\})$.
Then
$\gamma_n$ and $\delta_n$ are less than $\al$, so 
apply the induction hypothesis to obtain $Y_n\in\mathcal{E}_{B_{a_n}}$ such that $Y_n\le X_{w_u\cup\{n\}}$ and
one of  $P|Y_n\tl Q|Y_n$, $P|Y_n= Q|Y_n$, $P|Y_n\tr Q|Y_n$ holds.
Let $M$ be an infinite subset of $N$ such that the same one of the  three possibilities holds for each $n\in M$.
Apply the Fusion Lemma \ref{lem.fusion2} to obtain a $Z\in[u,X]$ such that 
for each $n\in\om$, if $Z_{w_u\cup\{n\}}\ne\emptyset$,
then $n\in M$ and $Z_{w_u\cup\{n\}}\le Y_n$ in $\mathcal{E}_{B_{a_n}}$.
Then $Z$ satisfies the lemma.
\end{proof}

\begin{defn}[Canonical Projection Maps]\label{defn.uniformproj}
For $X\in\mathcal{E}_B$,  $w_*\in\hat{X}\setminus X$,
and $P\in \UP(B,X,w_*)$,
define  the {\em projection map} 
$\pi_P:X_{w_*}\ra P$ as follows.
If $P=\{w_*\}$, then define $\pi_P(w)=\emptyset$ for all $w\in X_{w_*}$.
Otherwise,
for $w\in X_{w_*}$, let $\pi_P(w)$ be {\em the} $p\in P$ such that $p\tle w$.
The set $\{\pi_P:P\in\UP(B,X,w_*)\}$
is
the collection of {\em canonical projections} on $X_{w_*}$.
\end{defn}

\begin{defn}[Canonical Equivalence Relations on  $X_{w_*}$]\label{def.canon1ext}
Given $P\in\UP(B,X,w_*)$, define the equivalence relation
$E_P$ on  $X_{w_*}$ 
by $w\ E_P\ w'$ if and only if $\pi_P(w)=\pi_P(w')$.
\end{defn}

Since for $X\in\mathcal{E}_B$, $\mathcal{AE}^B_1|X=\{\{w\}:w\in X\}$, an equivalence relation on  $\mathcal{AE}^B_1|X$ is essentially the same as an equivalence relation on $X$, and vice versa.
Thus, we may  consider the members of  $\UP(B,X):=\UP(B,X,\emptyset)$ as  canonical projections on $X$.

For $\{n\}\in\rho_B^{-1}(X)$,
let $w_n=\rho_B(\{n\})$ and 
let $\UP(B_n,X_n)$ denote the appropriate shift of $\UP(C,Y)$, where $C=\delta_n(B_n\setminus\{n\})$ and $Y=
  \rho_C(\delta_n( \rho_B^{-1}(X_n)\setminus\{n\}))$.
Precisely,
$\UP(B_n,X_n)$ is the set of all
  $\rho_B( \{n\}\cup (\rho_C^{-1}(Q)+(n+1))$, for  $Q\in\UP(C,Y)$.

\begin{thm}[Canonization for $\mathcal{AE}^B_1$]\label{thm.CanonicalonAE_1}
Let $u\le_{\fin} X\in\mathcal{E}_B$, and let $E$ be an equivalence relation on the members of $X$.
Then there is a $Z\le X$ and a $P\in\UP(B,X)$ such that $E|Z=E_P|Z$.
\end{thm}

\begin{proof}
The proof proceeds by induction on the rank of $B$, using the fact that, 
if $C =\{ \delta_n(b\setminus \{n\}):b\in B_n\}$,
then there is an isomorphic correspondence between $\mathcal{E}_C$ and $\mathcal{E}_{B_n}$.
If $\rank(B)=1$, then $\mathcal{E}_B$ is the Ellentuck space.
Given any equivalence relation $E$ on some $X\in\mathcal{E}$,
the \Erdos-Rado Theorem in \cite{Erdos/Rado50} implies there is a $Y\le X$ such that 
either for all $w,w'\in Y$, $w\, E\, w'$, or else for all $w,w'\in Y$,
$w \not\hskip-.06in E\, w'$.
In the first case, $P=\{\emptyset\}$, and in the second case, $P=Y$ provide the uniform projections in $\UP(B,Y)$ such that $E_P$ canonizes $E$ on $Y$.

Now suppose that for all $C$ with $\rank(C)<\al$, the lemma holds, and suppose that $\rank(B)=\al$.
Let $X\in\mathcal{E}_B$ and $E$ be an equivalence relation on $X$.
Let $\mathcal{H}=\{\fa_2^B(Y):Y\le X\}$, the set of all second full approximations of members below $X$.
Note that for $x\in \mathcal{H}$,
$x(0)\wedge x(|x|-1)=\emptyset$, and moreover, for any $w,w'\in X$ with $w\wedge w'=\emptyset$, there is an $x\in \mathcal{H}$ such that $\{w,w'\}=\{x(0),x(|x|-1)\}$.

Let $\mathcal{H}_0=\{x\in\mathcal{H}: x(0)\, E\, x(|x|-1)\}$, and apply the Abstract Nash-Williams Theorem to obtain  $Y\le X$ homogeneous for $\mathcal{H}_0$.
Suppose that $\mathcal{H}|Y\sse\mathcal{H}_0$,
and let $w,w'\in Y$.
Take $w''\in Y$ so that $\min(w'')>\max(w\cup w')$.
Then there are $x,y\in\mathcal{H}|Y$ such that $x(0)=w$, $y(0)=w'$, and $x(|x|-1)=y(|y|-1)=w''$.
Since $x(0)\, E\, x(|x|-1)$ and $y(0)\, E\, y(|y|-1)$,
it follows that $w\, E\, w'$.
Therefore, each pair of members of $Y$ are equivalent, and $P=\{\emptyset\}$ is the uniform projection canonizing $E|Y$.

Now suppose that $\mathcal{H}|Y\cap \mathcal{H}_0=\emptyset$.
It follows that whenever $w,w'\in Y$ with $w\cap w'=\emptyset$, 
then $w\not \hskip-.05in E\, w'$.
Let $w_n$ denote $\rho_B(\{n\})$ and
 $Y_n$ denote $\{w\in Y:  w\tre w_n\}$, and
let $N\sse \om$ be the set of all $n$ such that
$Y_n\ne\emptyset$ (equivalently, $Y_n\in\mathcal{E}_{B_n}$).
By the induction hypothesis,
for each $n\in N$, there is a uniform projection $P_n\in\UP(B_n,Y_n)$
such that for any two $w,w'\in Y_n$, $w\, E\, w'$ if and only if $\pi_{P_n}(w)=\pi_{P_n}(w')$.

There is an infinite subset $N'\sse N$ such that 
either 
for all $m,n\in N$,
$m<n$ in $N'$ implies $\rank(P_m)=\rank(P_n)$,
or else 
for all $m,n\in N$,
$m<n$ in $N'$ implies 
 $\rank(P_m)<\rank(P_n)$.
Apply the Fusion Lemma \ref{lem.fusion1}
to obtain a $Z\le X$ such that whenever $Z_n\ne\emptyset$, then 
$n\in N'$ and $Z_n\sse Y_n$.
Let $P=\bigcup\{P_n:n\in N'$ and $Z_n\ne \emptyset\}$.
Then $P$ is in $\UP(B,Z)$.

We claim that $\pi_P$ canonizes $E|Z$.
Let $w,w'\in Z$.
Notice  that for each $n\in N'$, $w_n\tle P_n$; so
if
$w\cap w'=\emptyset$, then $\pi_P(w)\ne\pi_P(w')$ and also $w\not\hskip-.05in E\, w'$.
If $w\cap w'\ne\emptyset$, then let $n$ be such that $w_n\tle w,w'$.
Then $w\, E\, w'$ if and only if $\pi_{P_n}(w)=\pi_{P_n}(w')$ if and only if $\pi_P(w)=\pi_P(w')$.
Therefore, $P$ canonizes $E|Z$.
\end{proof}

\begin{notation}\label{notn.X_uetc}
Given $u\le_{\fin}X\in\mathcal{E}_B$,
let $w_u$ denote the $\tl$-maximum in $\hat{X}$ such that $w_u\tl v(|u|)$ for every $v\in r_{|u|+1}[u,X]$.
Let $X_{w_u}$ denote $\{w\in X:w\tr w_u\}$ and
 $X_u$ denote $\{w\in X:u\cup \{w\}\in r_{|u|+1}[u,X]\}$.
Note that $X_u=\{w\in X_{w_u}: \min (w\setminus w_u)>\max(u)\}$, which we also denote as $X_{w_u}/u$.
\end{notation}

This section concludes with the next theorem, laying some  groundwork for the Ramsey-classification Theorem in the  next section.

\begin{thm}[Canonization for 1-Extensions]\label{thm.canonAR1}
Let $u\le_{\fin}X\in\mathcal{E}_B$,
 and let  $E$ be an equivalence relation on 
$r_{|u|+1}[u,X]$.
Then there is a $Z\in [u, X]$
and a $P\in \UP(B,Z,u)$
 such that  
$E$ is canonical on 
$r_{|u|+1}[u,Z]$.
\end{thm}

\begin{proof}
Let  $a$ denote $\rho^{-1}_B(w_u)$.
By  Theorem \ref{thm.CanonicalonAE_1},
there is a $Y_{w_u}\in\mathcal{E}_{B_{a}}$ such that $Y_{w_u}\sse X_u$ and a $Q\in\UP(B_{a},Y_{w_u})$ such that $E\re Y_{w_u}=E_Q\re Y_{w_u}$.
Fusion Lemma \ref{lem.fusion2} implies that 
there is a $Z\in [u,X]$ such that $Z_u\sse Y_{w_u}$.
Then $Q\cap Z_u$ is in $\UP(B,Z,u)$ and canonizes $E$ on $r_{|u|+1}[u,Z]$.
\end{proof}


\section{Ramsey-classification Theorems}\label{sec.rct}

The main theorem of this section is the Ramsey-classification Theorem \ref{thm.rc} showing that for each uniform barrier $B$, any given equivalence relation on a front on $\mathcal{E}_B$  is canonized by
a projection map which is Nash-Williams and further satisfies a certain property $(*)$,
when restricted to some member of $\mathcal{E}_B$ 
  (see Definitions \ref{def.innerNW} and \ref{defn.canon} below).
A projection map projects  each member  $u\in \mathcal{AE}^B$ 
to some subtree of $\hat{u}$.  
The main theorem extends Theorem 33 of \cite{DobrinenJSL15}, which in turn extends the \Pudlak-Rodl\ Theorem  in the following sense:
 If $C\tl B$ and $R$ is an equivalence relation on a front $\mathcal{F}$ consisting of full approximations on $\mathcal{E}_C$, then 
taking $\mathcal{F}'$ to be front  of all full approximations on $\mathcal{E}_B$ extending members of $\mathcal{F}$ and taking $R'$ to be the induced equivalence relation on $\mathcal{F}'$,
 the canonical map for $R'$ on $\mathcal{F}'|X$ projects to a canonical map for $R$ on $\mathcal{F}|\pi_{C,B}(X)$.

The following notation will be used throughout.

\begin{notation}\label{defn.ext}
Let $X\in\mathcal{E}_B$, $u,v\in\mathcal{AE}^B$, and $k,n<\om$.
We shall leave the superscript $B$ off of $r^B_k(X)$ whenever no confusion will arise.
 $X(n)$  denotes the $\prec$-$n$-th member of  $X$, and $r_k(X)=\{X(n):n<k\}$.
Let
$\Ext(u,v)=\{Y\in\mathcal{E}_B:u\le_{\fin} Y\}$,
$\Ext(u,X)=\{Y\le X: u\le_{\fin}Y\}$,
and  $\Ext(u,v,X)=\Ext(u,X)\cap\Ext(v,X)$.

Let
$X/u=\{X(n): n<\om$ 
and   $\max X(n)>\max u\}$ and 
$v/u=\{v(n): n<|v|$ and $\max v(n)> \max u\}$.
Let $[v,X/u]$ denote $\{Y\in\mathcal{E}_B:v\sqsubset Y$ and $Y/v\sse X/u\}$.
Let $r_k[u,X/v]$ denote $\{z\in\mathcal{AE}^B_k:z\sqsupseteq u$ and $z/u\sse X/v\}$, and
 $r[u,X/v]$ denote  $\bigcup\{r_k[u,X/v]:  k\ge |u|\}$.
Let $\depth_X(u,v)=\max \{\depth_X(u),\depth_X(v)\}$.
\end{notation}

\begin{fact}\label{fact.useful}
Suppose  $Y\le X$ are in  $\mathcal{E}_B$
and $u,v\in\mathcal{AE}^B$ satisfy
 $v\le_{\fin} u$, $\max u=\max v$, $u\le_{\fin} X$, and $v\le_{\fin} Y$.
Then there is a $Z\in [u,X]$ such that
for each $y\le_{\mathrm{fin}} v$ and each $z\in r[y,Z/v]$, 
$z/v\sse Y$.
\end{fact}

\begin{proof}
Let $d=|u|$ and $z_d=u$.
 Given $z_n$ for $n\ge d$,
if $w_{z_n}\in \hat{Y}$,
then choose $z_{n+1}\in r_{n+1}[z_n,Y/u]$.
If $w_{z_n}\not\in\hat{Y}$,
then choose any  $z_{n+1}\in r_{n+1}[z_n,X]$.
It is straightforward to check that  $Z=\bigcup_{n\ge d}z_n$ 
satisfies the conclusion.
\end{proof}

Recall Definition \ref{def.frontR1} of front on a topological Ramsey space from Section \ref{sec.reviewtRs}.

\begin{defn}\label{defn.mix}
Let $\mathcal{F}$ be a  front on $\mathcal{E}_B$ and let  $R$ be an equivalence relation on $\mathcal{F}$.
Let $\hat{\mathcal{F}}=\{r_k(u):u\in\mathcal{F}$ and $k\le |u|\}$.
Suppose  $u,v\in\hat{\mathcal{F}}$ 
and $X\in\Ext(u,v)$.
We say that $X$ {\em separates} $u$ and $v$ if and only if
for all  $y\in\mathcal{F}\cap r[u,X/v]$ and $z\in \mathcal{F}\cap r[v,X/u]$, $y \not\hskip-.06in R\, z$.
We say that $X$ {\em mixes} $u$ and $v$ if and only if no $Y\in\Ext(u,v,X)$ separates $u$ and $v$.
We say that $X$ {\em decides} for $u$ and $v$ if and only if either $X$ mixes $u$ and $v$ or else $X$ separates $u$ and $v$.
\end{defn}

Note that $X\in\Ext(u,v)$  mixes $u$ and $v$ if and only if
 for each $Y\in\Ext(u,v,X)$, there are $y\in\mathcal{F}\cap r[u,Y/v]$ and $z\in\mathcal{F}\cap  r[v,Y/u]$ for which $y\, R\, z$.
The following  fact and lemma  follow from similar   proofs to those  of  Facts 24 and 25
in \cite{DobrinenJSL15}.
We do point out that the  proof of the transitivity of mixing is the one place in this article where  Abstract Ellentuck Theorem is used and the Abstract Nash-Williams Theorem does not seem to be sufficient.

\begin{fact}\label{fact.equivformsmix}
The following are equivalent for $X\in\Ext(u,v)$:
\begin{enumerate}
\item
$X$ mixes $u$ and $v$.
\item
For all $Y\in\Ext(u,v,X)$,
there are $y\in\mathcal{F}\cap r[u,Y/v]$ and $z\in\mathcal{F}\cap  r[v,Y/u]$ for which $y\, R\, z$.
\item
For all $Y\in[\depth_X(u,v),X]$, there are  $y\in\mathcal{F}\cap r[u,Y/v]$ and $z\in\mathcal{F}\cap  r[v,Y/u]$ for which $y\, R\, z$.
\end{enumerate}
\end{fact}

\begin{lem}[Transitivity of Mixing]\label{lem.transmix}
Suppose that $X$ mixes $u$ and $v$ and $X$ mixes $v$ and $z$.
Then $X$ mixes $u$ and $z$.
\end{lem}


Next, we define the notion of a hereditary property, and give   a general lemma about fusion  to obtain a member of $\mathcal{E}_B$ on which a hereditary property holds.

\begin{defn}\label{hered}
A property  $P(u,X)$ defined on $\mathcal{AE}^B\times\mathcal{E}_B$
 is  {\em hereditary} if whenever $X\in\Ext(u)$ and $P(u,X)$ holds, then also $P(u,Y)$ holds for all $Y\in[\depth_X(u),X]$.
Similarly, a property $P(u,v,X)$ defined on $\mathcal{AE}^B\times\mathcal{AE}^B\times\mathcal{E}_B$
 is {\em hereditary} if whenever $P(u,v,X)$ holds, then also $P(u,v,Y)$ holds for all $Y\in [\depth_X(u,v),X]$.
\end{defn}

\begin{lem}\label{lem.hered}
Let $P(\cdot,\cdot)$ be a hereditary property on $\mathcal{AE}^B\times\mathcal{E}_B$.
If whenever $u\le_{\fin}X$ there is a $Y\in[\depth_X(u),X]$ such that $P(u,Y)$,
then for each $Z\in\mathcal{E}_B$, there is a $Z'\le Z$ such that for all $u\in\mathcal{AE}^B|Z'$,
$P(u,Z')$ holds.

Likewise, suppose  $P(\cdot,\cdot,\cdot)$ is  a hereditary property on $\mathcal{AE}^B\times\mathcal{AE}^B\times\mathcal{E}_B$.
If whenever $X\in \Ext(u,v)$ there is a $Y\in[\depth_X(u,v),X]$ such that $P(u,v,Y)$,
then for each $Z\in\mathcal{R}$, there is a $Z'\le Z$ such that for all $u,v\in\mathcal{AE}^B|Z'$,
$P(u,v,Z')$ holds.
\end{lem}

The proof of Lemma \ref{lem.hered} is straightforward, using the Fusion Lemma \ref{lem.fusion2};  being very similar to that of  Lemma 4.6 in \cite{Dobrinen/Todorcevic14}, we omit it.

\begin{lem}\label{lem.decide}
Given any front $\mathcal{F}$ and function $f:\mathcal{F}\ra\om$,
there is an $X\in\mathcal{E}_B$ such that for all $u,v\in\hat{\mathcal{F}}|X$,
$X$ decides $u$ and $v$.
\end{lem}

Lemma \ref{lem.decide} follows immediately from Lemma \ref{lem.hered}  and the fact that mixing and separating are hereditary properties.

Given a set $T'$ of $\tle$-incomparable members of $\widehat{\bW}_B$,
we call $T$ a {\em  projection of $T'$} if 
$T$ is a collection  of $\tle$-incomparable members of $\hat{T}'$.
In other words, $T$ is the set of maximal nodes of some subtree of $\hat{T}'$.
$T$ is a {\em proper projection of $T'$} if $T$ is a projection of $T'$ which is not equal to $T'$.

\begin{defn}\label{def.innerNW}
We call a  map $\vp$ on a front $\mathcal{F}\sse \mathcal{AE}^B$ on $X$
\begin{enumerate}
\item
{\em inner} if for each $u\in \mathcal{F}$, 
$\vp(u)$ is a projection of $u$.
\item
{\em Nash-Williams} if whenever $u,v\in\mathcal{F}$ and there is some $n\le |u|$ such that $\vp(u)\cap\widehat{r_n(u)}=\vp(v)$,
then $\vp(u)=\vp(v)$.
\item 
{\em irreducible} if it is inner and 
the following property $(*)$ holds:
\begin{enumerate}
\item[$(*)$]
There is a $Y\le X$ such that
given $u\in\mathcal{F}|Y$, there is a $v\in\mathcal{F}$ such that \begin{equation}
\vp(u)=\vp(v)= u\wedge v,
\end{equation}
\end{enumerate}
where $u\wedge v$ denotes the set of maximal nodes in $\hat{u}\cap\hat{v}$.
\end{enumerate}
\end{defn}

\begin{defn}[Canonical equivalence relations on a front]\label{defn.canon}
Let $\mathcal{F}$ be a front on $\mathcal{E}_B$.
An equivalence relation $R$ on $\mathcal{F}$ is {\em canonical} if  and only if
there is an irreducible map $\vp$ canonizing $R$ on $\mathcal{F}$,
meaning that for all $u,v\in\mathcal{F}$, 
$u\, R\, v \llra \vp(u)=\vp(v)$.
\end{defn}

Irreducible maps on $\mathcal{E}_B$ 
are unique in the following sense.

\begin{thm}\label{thm.irred}
Let  $R$ be an equivalence relation on some front $\mathcal{F}$ on $\mathcal{E}_B$.
Suppose $\vp$ and $\vp'$ are irreducible maps  canonizing $R$.
Then there is an $X\in\mathcal{E}_B$ such that 
for each $u\in\mathcal{F}|X$,
$\vp(u)=\vp'(u)$.
\end{thm}

 Theorem \ref{thm.irred} will  be proved after  Theorem \ref{thm.rc}.

We now prove  the  Ramsey-classification theorem for equivalence relations on fronts.
Like the proof of the Ramsey-classification theorem for the spaces $\mathcal{E}_k$, $k<\om$, in \cite{DobrinenJSL15},
the proof  follows the same general outline as that  of Theorem 4.14 in \cite{Dobrinen/Todorcevic14}, and the facts which have proofs  similar to those in \cite{Dobrinen/Todorcevic14} will be stated here without proof.
However, the majority of the proof of this theorem involves new arguments,
required because of the infinitely splitting structure of the spaces and the infinite rank of the barriers.

\begin{thm}[Ramsey-classification Theorem]\label{thm.rc}
Let $B$ be a uniform barrier on $\om$.
Given $V\in\mathcal{E}_B$ and   an equivalence relation $R$ on a front $\mathcal{F}$ on $V$,
there is a $W\le V$ such that $R$ restricted to $\mathcal{F}|W$ is canonical.
\end{thm}

\begin{proof}
By Lemma \ref{lem.decide}, shrinking $V$ if necessary, we may assume  that for all $u,v\in \hat{\mathcal{F}}|V$,
$V$ decides for $u$ and $v$.
Recall that for any $u\le_{\fin} X\in\mathcal{E}_B$, 
mixing on $r_{|u|+1}[u,X]$ is an equivalence relation
 on $X_u:=\{y(|u|):y\in r_{|u|+1}[u,X]\}$.
For a uniform projection $P_u\in\UP(B,Y,u)$, let $\pi_{P_u}$ denote the projection map which takes $w\in Y_u$ and sends it to {\em the} $p\in P_u$ such that $p\tle w$.

\begin{claim}\label{claim.mix1}
Given $X\in\mathcal{E}_B$,
there is a $Y\le X$ such that for each $u\in (\hat{\mathcal{F}}\setminus\mathcal{F})|Y$,
there is a uniform projection $P_u\in\UP(B,Y,u)$ such that 
for all $y,z\in r_{|u|+1}[u,Y]$,
$Y$ mixes $y$ and $z$ if and only if $\pi_{P_u}(y(|u|))=\pi_{P_u}(z(|u|))$.
\end{claim}

\begin{proof}
Let $u\le_{\fin}Z\le X$ be given, let $n=|u|$, and
let $E_u$ denote the equivalence relation on $Z_u$
induced by the equivalence relation of 
 mixing on 
$ r_{|u|+1}[u,Z]$.
By the Canonization Theorem for 1-Extensions \ref{thm.canonAR1},
there is a $Z'\in [u,Z]$ and a $P_u\in\UP(B,Z',u)$
 such that  $E_u|Z'_u$ is canonical, given by $\pi_{P_u}$.
Since mixing is a hereditary property,
Lemma \ref{lem.hered}
implies there is a $Y\le X$ such that for all $u\le_{\fin} Y$,
   $E_u$ is canonical on  $Y_u$.
\end{proof}

Thus, by shrinking $V$ if necessary, we may also assume that $V$ satisfies Claim \ref{claim.mix1}.
For  $u\in \mathcal{AE}^B|V$, 
let $P_u$ denote the uniform projection
$P\in\UP(B,V,u)$ given by Claim \ref{claim.mix1} which determines the mixing relation on $r_{|u|+1}[u,V]$,
and let $\pi_u$ denote $\pi_{P_u}$.

\begin{defn}[The maps $\vp$ and $\vp'$]\label{def.proj}
For $v\in \hat{\mathcal{F}}|V$, 
define 
$\vp(v)$
to be the set of all $\tle$-maximal nodes in $\{\pi_{r_n(v)}(v(n)):   n<|v|\}$; and define $\vp'(v)$ to be the set $\{\pi_{r_n(v)}(v(n)):   n<|v|\}$.
\end{defn}

Thus, $\vp$ is by definition an inner map.

\begin{rem}
We shall work mostly with $\vp'$ and show that it canonizes the equivalence relation $R$ on some $W$.
Toward the end of the proof of this theorem,
we shall  show that 
 for all $u,v\in\hat{\mathcal{F}}|W$,  $\vp(u)=\vp(v)$ if and only if $\vp'(u)=\vp'(v)$.
\end{rem}

The next fact is straightforward, its  proof   closely resembling  that of  Claim 4.17 in \cite{Dobrinen/Todorcevic14}.


\begin{fact}\label{claim.4.17}
Suppose $u\in(\hat{\mathcal{F}}\setminus\mathcal{F})|V$ and $v\in\hat{\mathcal{F}}|V$.
\begin{enumerate}
\item
Suppose $u\in\mathcal{AE}^B_{n}|V$  and  $y,z\in r_{n+1}[u,V]$.
If $V$ mixes $y$ and $v$ and $V$ mixes $z$ and $v$, then $\pi_u(y(n))=\pi_u(  z(n))$.
\item
If $u\sqsubset v$ and $\vp'(u)\setminus\{\emptyset\}=\vp'(v)\setminus\{\emptyset\}$,
then $V$ mixes $u$ and $v$.
\end{enumerate}
\end{fact}

Recall that $X_u/v$ denotes $\{y(m):y\in r_{m+1}[u,X/v]\}$, which is the same as
$\{w\in X:w\tr w_u$ and $\min(w\setminus w_u)>\max(u\cup v)\}$.

\begin{fact}\label{fact.incompiffdisjt}
Let $u,v\in\mathcal{AE}^B|V$,  $m=|u|$ and $n=|v|$.
If $w_u=w_v$, then 
$V_u/v=V_v/u$.
If  $w_u\ne w_v$, then
$V_u/v$ and $V_v/u$ are disjoint.
\end{fact}

\begin{proof}
If $w_u=w_v$, then it is clear that $V_u/v=\{w\in  V:w\tr w_u$ and $\min(w\setminus w_u)>\max(u\cup v)\}
=V_v/u$.

If $w_u$ and $w_v$ are $\tle$-incomparable, then 
$V_u/v\sse\{w\in V:w\tr w_u\}$ and 
$V_v/u\sse \{w\in V:w\tr w_v\}$, so they are disjoint.
If   $w_u\tl w_v$,
then for each $y\in r_{m+1}[u,V/v]$,  
$\min(y(m)\setminus w_u)>\max(v)>\max(w_v)$, and hence $y(m)\not\tr w_v$.
Likewise, if $w_v\tl w_u$, then each $z\in r_{n+1}[v,V/u]$ 
satisfies
 $z(n)\not\tr w_u$.
Thus, $w_u\ne w_v$ implies that
$V_u/v$ and 
$V_v/u$ are disjoint.
\end{proof}

\begin{lem}\label{lem.miximpliesp=q}
For each pair $u,v\le_{\fin} V$,
if $P_u\ne\{w_u\}$, $P_v\ne\{w_v\}$ and  $w_u\ne w_v$, then $V$ separates $u$ and $v$.
\end{lem}

\begin{proof}
Let $d_0=\depth_{V}(u,v)$ and let $d\ge d_0$ be the  least integer such that $w_{r_d(V)}$ equals $w_u$ or  $w_v$.
Without loss of generality, assume that $w_{r_d(V)}=w_u$.
Then
 for each $X\in [d,V]$, $X(d)\tr w_u$ and $X(d)\not\tr w_v$, since $w_v\ne w_u$.
Let $\mathcal{H}$ be the collection of all $x\in \mathcal{AE}^B|V$ such that 
\begin{enumerate}
\item
$x\sqsupset r_d(V)$; 
\item
$x\cap V_v/u$ is a second full approximation in $\mathcal{E}_{B_{a_v}}$; and 
\item
$|x|$ is minimal with respect to (2).
\end{enumerate}
In other words, (2) and (3) state that  $\{w\in x:w\tr w_v$ and $\min(w\setminus w_v)>\max (r_d(V))\}=\fa_2^{B_{a_v}}(Z)$ for some $Z\in\mathcal{E}_{B_{a_v}}$, 
(since for $w\tr w_v$, $\min(w\setminus w_v)>\max (r_d(V))$ if and only if $\min(w\setminus w_v)>\max (u)$).
For $x\in\mathcal{H}$, let $k_x<l_x$ denote the numbers such that $x(k_x)$ is minimal in $x\cap (V_v/u)$ and $x(l_x)$ is maximal in $x\cap (V_v/u)$.
Then in fact $l_x=|x|-1$ and $\pi_v(x(k_x))\ne\pi_v(x(l_x))$.
Notice that $\mathcal{H}$ is a front on $[d,V]$.

Let $\mathcal{H}_0$ denote the set of those $x\in \mathcal{H}$ such that 
$V$ mixes $u\cup x(d)$ and $v\cup x(k_x)$.
Apply the Abstract Nash-Williams Theorem to obtain an $X\in [d,V]$ homogeneous for the partition of $\mathcal{H}$ induced by $\mathcal{H}_0$.
If $\mathcal{H}|X\sse\mathcal{H}_0$,
then in fact $V$ also mixes $u\cup x(d)$ and $v\cup x(l_x)$ for each $x\in \mathcal{H}|X$, since $x(l_x)$ is  the last member of some  full approximation in $\mathcal{E}_{B_{a_v}}$ and hence is equal to $y(k_y)$ for some $y\in\mathcal{H}|X$.
Then taking any $x\in\mathcal{H}|X$,  we have that $\pi_v(k_x)\ne\pi_v(l_x)$, $V$ mixes $u\cup x(d)$ and $v\cup x(k_x)$, and $V$ mixes $u\cup x(d)$ and $v\cup x(l_x)$.
Transitivity of mixing implies that $V$ mixes $v\cup x(k_x)$ and $v\cup x(l_x)$, contradicting that $\pi_v(x(k_x))\ne\pi_v(x(l_x))$.
Therefore, it must be the case that $(\mathcal{H}|X)\cap \mathcal{H}_0=\emptyset$.
For each pair $w\in X_u/v$ and $w'\in X_v/u$ with $w=X(m)$ and $w'=X(n)$ for some $m<n$,
 there is an $x\in\mathcal{H}|X$ such that $x(d)=w$ and $x(k_x)=w'$.
Hence, $V$ separates $u\cup w$ and $v\cup w'$.

To take care of the  case when 
$w\in X_u/v$ and $w'\in X_v/u$ with  $w=X(n)$ and $w'=X(m)$ for some $n>m$,
let $\mathcal{K}$ be the collection of all $x\in\mathcal{AE}^B|X$ such that 
\begin{enumerate}
\item
$x\sqsupset r_d(V)$;
\item
$x(|x|-1)\tr w_u$ and there are $i_x<j_x \in (d,|x|-1)$ such that $x(i_x)\tr w_v$ and $x(j_x)\tr w_v$; and
\item
$|x|$ is minimal with respect to (2).
\end{enumerate}
Let $\mathcal{K}_0$ consist of those $x\in \mathcal{K}$ for which $V$ mixes $u\cup x(|x|-1)$ and $v\cup x(i_x)$.
Apply the Abstract Nash-Williams Theorem to obtain an $X'\in[d,X]$ homogeneous for the partition of $\mathcal{K}$ induced by $\mathcal{K}_0$.
Since,  $\pi_v(x(i_x))\ne\pi_v(x(j_x))$ for each $x\in\mathcal{K}$,
a similar argument to the one above yields that 
 $(\mathcal{K}_0|X')\cap\mathcal{K}=\emptyset$.
Hence, $X'$ separates $u$ and $v$, which implies the Lemma, since $V$ already decided $u$ and $v$.
\end{proof}

\begin{lem}\label{lem2.5}
Let $u,v\in\mathcal{AE}^B|V$.
Suppose that $V$ mixes $u$ and $v$, and both
 $P_u\ne\{w_u\}$ and $P_v\ne\{w_v\}$.
Then for each $X\in\Ext(u,v,V)$, there is a $Y\in[d,V]$
such that 
$P_u|Y_u/v=P_v|Y_v/u$, where $d\ge \depth_X(u,v)$ is least such that $w_{r_d(X)}=w_u$.
\end{lem}

\begin{proof}
The hypotheses along with Lemma \ref{lem.miximpliesp=q} imply that $w_u=w_v$, which we denote by $w_*$.
Let $X\in\Ext(u,v,V)$ be given, and let $d \ge \depth_X(u,v)$ be least such that $w_{r_d(X)}=w_*$.
Let $a_*$ denote $\rho^{-1}_B(w_*)$.
Without loss of generality, suppose that $\depth_X(v)=d$.
Then  $X_u/v=X_v/u=X_v$, which we shall denote as $X_*$.
By Lemma \ref{lem.CsqsubD} and the Fusion Lemma \ref{lem.fusion2},
there is an $X'\in [d,X]$ such that 
$P_u|X'_u/v$ and  $P_v|X'_v/u$ are 
related by exactly one of $\tl$, $=$, or $\tr$.

Let $X'_*$ denote $X'\cap X_*$.
Let $\mathcal{H}$ be the collection of all $x\in\mathcal{AE}^B|X'$ such that 
\begin{enumerate}
\item
$x\sqsupset r_d(X)$;
\item
$x\cap X'_*$ is a second full approximation in $\mathcal{E}_{B_{a_*}}$; and 
\item
$|x|$ is minimal with respect to (2).
\end{enumerate}
Then $\mathcal{H}$ is a front on $[d,X']$.
For $x\in\mathcal{H}$, let $k_x$ denote the index of the least member of $x\cap X'_*$.
Let $\fa_1^*(x)$ denote $\{x(k_x)\}$ and 
$\fa_2^*(x)$ denote  $x\cap X'_*$; these are  the  exactly first and second full approximations of $x\cap X'_*$ in $\mathcal{E}_{B_{a_*}}$, respectively.

Let $\mathcal{H}_0$ be the subset of those $x\in\mathcal{H}$ such that $V$ mixes  $u\cup x(k_x)$ and $v\cup x(k_x)$.
Let $\mathcal{H}_1$ be the collection of  those $x\in\mathcal{H}$ such that for some $w\in\fa^*_2(x)\setminus \fa^*_1(x)$,
$V$ mixes $u\cup x(k_x)$ and $v\cup w$.
Let $\mathcal{H}_2$  the collection of those $x\in\mathcal{H}$ such that for some $w\in\fa^*_2(x)\setminus \fa^*_1(x)$,
$V$ mixes $v\cup x(k_x)$ and $u\cup w$.
Applying the Abstract Nash-Williams Theorem three times, obtain a $Y\in[d,X']$ which is homogeneous for all three partitions of $\mathcal{H}$ induced by the $\mathcal{H}_i$, $i<3$.

The rest of the proof of this lemma proceeds by showing that for each $i<3$, it must be the case that 
 $(\mathcal{H}|Y)\cap \mathcal{H}_i=\emptyset$, 
and then arguing that this implies
$V$ separates $u$ and $v$, contradicting the hypothesis of the lemma.
This will lead to the conclusion that 
$P_u|X'_u/v\not\tl P_v|X'_v/u$.
The same argument, interchanging $u$ and $v$, yields that
$P_u|X'_u/v\not\tr P_v|X'_v/u$.
Therefore, the only possibility is $P_u|X'_u/v=P_v|X'_v/u$, which proves the lemma.

Let $x,y\in\mathcal{H}|Y$ such that $\pi_u(x(k_x))=\pi_u(y(k_y))$ but $\pi_v(x(k_x))\ne\pi_v(y(k_x))$.
Then $V$ separates $v\cup x(k_x)$ and $v\cup y(k_y)$.
Now if $\mathcal{H}|Y\sse \mathcal{H}_0$, then 
 $V$ mixes $u\cup x(k_x)$ and $u\cup y(k_y)$, $V$ mixes $u\cup x(k_x)$ and $v\cup x(k_x)$, and $V$ mixes $u\cup y(k_y)$ and $v\cup y(k_y)$;
hence
$V$ mixes $v\cup x(k_x)$ and $v\cup y(k_y)$, a contradiction.
Thus, $(\mathcal{H}|Y)\cap \mathcal{H}_0=\emptyset$.

Next, suppose that $\mathcal{H}|Y\sse\mathcal{H}_1$.
Fix any  $x\in\mathcal{H}|Y$. 
Recall  that $x(k_x)$ determines the number $p$ such that 
every second full approximation in  $\mathcal{AE}_{B_{a_*}}$ extending $x(k_x)$ has cardinality $p$.
For $y\in\mathcal{H}|Y$ such that $y\sqsupset r_{k_x+1}(x)$,
$y_*$ is a second full approximation in $\mathcal{AE}_{B_{a_*}}$; hence
  $\{y_*(0),\dots,y_*(p)\}=y_*$.

Fix $r_{k_x+1}(x)$, and let $\mathcal{K}=\{y\in\mathcal{H}|Y:y\sqsupset  r_{k_x+1}(x)\}$.
Thus, $\mathcal{K}$ is a front on $[r_{k_x+1}(x),Y]$.
For each $j<p$, let 
 $\mathcal{K}_j$ be the collection of $y\in\mathcal{H}|Y$ such that 
$y\sqsupset r_{k_x+1}(x)$ and $V'$ mixes $u\cup x(k_x)$ and $v\cup y_*(j+1)$.
Apply the Abstract Nash-Williams Theorem $p$ times to obtain a $Y'\in [r_{k_x+1}(x),Y]$ which is homogeneous for each of the partitions of $\mathcal{K}$ induced by $\mathcal{K}_j$, $j<p$.

Since $\mathcal{K}|Y'\sse\mathcal{H}|Y\sse\mathcal{H}_1$,
for each $y\in\mathcal{K}|Y'$, there is some $j<p$ for which $u\cup x(k_x)$ and $v\cup y_*(j+1)$ are mixed by $V$.
Fixing any such pair $y,j$,
it follows that $\mathcal{K}|Y'\sse\mathcal{K}_j$, since 
$\mathcal{K}|Y'$ is homogeneous for the partition $\mathcal{K}_j$, $\mathcal{K}\setminus \mathcal{K}_j$.
Let $w_{\wedge}$ denote the meet of $x(k_x)$ and $y_*(j+1)$.
Then $w_{\wedge}$ is the meet of $x(k_x)$ and $z_*(j+1)$ for each $z\in\mathcal{K}|Y'$.
There are two cases to consider.

\it Case 1.
\rm
$w_{\wedge}\tl \pi_v(x(k_x))$.
Since $P_u|Y_*\tl P_v|Y_*$ and $w_{\wedge}\tl y_*(j+1)$,
it follows that $\pi_v(z_*(j+1))\ne\pi_v(x(k_x))$, for each $z\in\mathcal{K}|Y'$.
Thus, there is a $z\in\mathcal{K}|Y'$ such that $\pi_v(z_*(j+1))\ne\pi_v(y_*(j+1))$.
Then $\mathcal{K}\setminus \mathcal{K}_j$ implies that $V$ mixes 
$u\cup x(k_x)$ and $v\cup y_*(j+1)$, and $V$ mixes $u\cup x(k_x)$ and $v\cup z_*(j+1)$.
But this implies that $V$ mixes $v\cup y_*(j+1)$ and $v\cup x_*(j+1)$, contradicting $\pi_v(y_*(j+1))\ne\pi_v(z_*(j+1)$.

\it Case 2.
\rm
$w_{\wedge}\tre \pi_v(x(k_x))$.
Then $\pi_v(y_*(j+1))=\pi_v(x(k_x))$, so $V$ mixes $v\cup x(k_x))$ and $v\cup y_*(j+1)$.
At the same time,
$\mathcal{K}\setminus \mathcal{K}_j$ implies that $V$ mixes 
$u\cup x(k_x)$ and $v\cup y_*(j+1)$.
Hence, $V$ mixes $u\cup x(k_x)$ and $v\cup x(k_x)$.
But this implies that $\mathcal{K}|Y'\sse \mathcal{H}_0$, whereas we already showed that $(\mathcal{H}|Y)\cap\mathcal{H}_0=\emptyset$, a contradiction.

Since in both cases we reach a contradiction, it must be the case that $(\mathcal{H}|Y)\cap \mathcal{H}_1=\emptyset$.
A similar argument, shows that $(\mathcal{H}|Y)\cap\mathcal{H}_2$ must be empty.
Thus, for each $i<3$, $(\mathcal{H}|Y)\cap \mathcal{H}_i$ is empty.
But this implies that $Y$ separates $u$ and $v$, a contradiction.
Therefore, 
the assumption that $P_u|X'_u/v\tl P_v|X'_v/u$ was incorrect.
As stated above, this leads to the conclusion of the lemma.
\end{proof}

The next lemma is the crux of the proof of the theorem.

\begin{lem}\label{lem.phiknows}
There is a $W\le V$ such that 
for all $u,v\in(\hat{\mathcal{F}}\setminus\mathcal{F})|W$  which are mixed by $W$, the following holds:
For all $y\in r_{|u|+1}[u,W/v]$ and $z\in r_{|v|+1}[v,W/u]$,
$W$ mixes $y$ and $z$ if and only if $\pi_u(y(|u|))=\pi_v(z(|v|))$.
\end{lem}

\begin{proof}
We will show that for all pairs $u,v\in (\hat{\mathcal{F}}\setminus\mathcal{F})|V$
which are  mixed by $V$,
for each $X\in\Ext(u,v,V)$,
there is a $Z\in[\depth_X(u,v),X]$ such that 
for all  $y\in r_{|u|+1}[u,Z/v]$ and $z\in r_{|v|+1}[v,Z/u]$,
$V$ mixes $y$ and $z$ if and only of $\pi_u(y(|u|))=\pi_v(z(|v|))$.
The lemma then follows from  Fact \ref{fact.equivformsmix} and Lemma \ref{lem.hered}.

Suppose
$u,v\in(\hat{\mathcal{F}}\setminus\mathcal{F})|V$
are mixed by $V$.
Let $m=|u|$ and  $n=|v|$.

\begin{claim}\label{claim.bothE_0}
$P_u=\{w_u\}$ if and only if $P_v=\{w_v\}$.
\end{claim}

\begin{proof}
Suppose toward a contradiction that $P_u=\{w_u\}$ but $P_v\ne\{w_v\}$.
Fact \ref{claim.4.17} (1) implies 
that there is at most one $E_v$ equivalence class of 
members $z$ of $r_{n+1}[v,V/u]$
for which $z$ is mixed with each member of $r_{m+1}[u,V/v]$.
If each $z\in r_{n+1}[v,V/u]$ is not mixed with any $y\in r_{m+1}[u,V/z]$,
then $V$ separates $u$ and $v$, a contradiction.
So, suppose  $z\in r_{n+1}[v,V/u]$ is mixed with some $y\in r_{m+1}[u,V/v]$.
By Fact \ref{claim.4.17} (2), every two members $y,y'$ of $r_{m+1}[u,V/v]$ are mixed, since  $P_u=\{w_u\}$.
Hence, $V$ mixes $z$ with every $y\in r_{m+1}[u,V/v]$.
Let $d'=\depth_V(u,v)$, and 
take $Z\in [d',V]$ such that 
the minimum number in $\bigcup Z \setminus \bigcup r_{d'}(V)$ is larger than  $\max (z)$.
Then for each $z'\in r_{n+1}[v,Z/u]$, $\pi_v(z')>\pi_v(z)$, so $z'  \hskip-.03in\not\hskip-.05in E_v\, z$.
Hence, $z'$
is separated from each $y\in r_{m+1}[u,Z/v]$. 
But this contradicts that $V$ mixes $u$ and $v$.
Therefore, $P_v=\{w_v\}$.
\end{proof}

If
  $P_u=\{w_u\}$ and  $P_v=\{w_v\}$,  then given any
$X\in\Ext(u,v,V)$,
 for all 
 $y\in r_{m+1}[u,X/v]$ and 
$z\in r_{n+1}[v,X/u]$,
$V$ mixes $y$ and $z$, by Fact \ref{claim.4.17} (2) and transitivity of mixing. 
At the same time, $\pi_u(y(m))=\pi_v(z(n))=\emptyset$.
In this case simply let $Z=X$.

For the rest of the proof of this lemma, assume that $P_u\ne\{w_u\}$ and $P_v\ne\{w_v\}$.
Let  $X\in\Ext(u,v,V)$ be given and
let $d=\depth_X(u,v)$.
Since we are assuming that $V$ mixes $u$ and $v$,
 Lemma \ref{lem.miximpliesp=q} implies that $w_u=w_v$, which we shall simply denote by $w_*$.
Let $X_*$ denote $X_u/v$, which, we point out, is equal to  $X_v/u$, and 
let  $a_*$ denote $\rho^{-1}_B(w_*)$.
By Lemma \ref{lem2.5},
there is an $X'\in [d,X]$ such that $P_u|X'_u/v=P_v|X'_v/u$.
Without loss of generality, assume $X$ already satisfies this, and
let $P$ denote $P_u|X_*$.

\begin{claim}\label{claim.mixiff=}
There is a $Z\in [d,X]$ such that for all $y\in r_{m+1}[u,Z/v]$ and $z\in r_{n+1}[v,Z/u]$,
$Z$ mixes $y$ and $z$ if and only if $\pi_u(y(m))=\pi_v(z(n))$.
\end{claim}

\begin{proof}
Let $d_0\ge d$ be least such that  $X(d_0)\in X_*$. 
Then 
$w_{r_{d_0}(X)}=w_*$ and 
$\{Y(d_0):Y\in [d_0,X]\}= X_*$.
For $x\in\mathcal{AE}^B|X$, let $x_*$ denote $x\cap X_*$.
Let $\mathcal{K}$ be the set of those $x\in\mathcal{AE}^B|X$ such that 
$x\sqsupset r_{d_0}(X)$, $x_*$ is a second full approximation in $\mathcal{E}_{B_{a_*}}$,
and $x(|x|-1)\in x_*$.
For $x\in\mathcal{K}$, let $l(x_*)$ denote the cardinality of $x_*$; hence $x_*=\{x_*(l):l<l(x_*)\}$.
Let $\mathcal{K}_0$ consist of those $x\in\mathcal{K}$ such that there is an $l< l(x_*)$ such that either
(a)
$\pi_P(x_*(0))=\pi_P(x_*(l))$ and $V$ separates $u\cup x_*(0)$ and $v\cup x_*(l)$;
or
(b)
$\pi_P(x_*(0))\ne\pi_P(x_*(l))$ and $V$ mixes $u\cup x_*(0)$ and $v\cup x_*(l)$.
Take $X'\in[d_0,X]$ homogeneous for the partition of $\mathcal{K}$ induced by $\mathcal{K}_0$.

Construct $Z\in [d_0,X']$ inductively as follows.
Noting that $w_{r_{d_0}(X')}=w_*$, let $X_0=X'$ and let $w_0$ denote $X_0(d_0)$.
For the inductive step, suppose we are given $w_i=X_i(d_i)\in X_*$, and
 let $l(w_i)$ denote the cardinality of any (hence every) second full approximation in $\mathcal{E}_{B_{a_*}}$ with $w_i$ as its least member.
Let
$\mathcal{K}^{w_i}$ denote the set of all $x\in\mathcal{K}|X'$ such that $x_*(0)=w_i$.
For each $l<l(w_i)$, let
$\mathcal{K}_l^{w_i}$ denote the set of all $x\in\mathcal{K}^{w_i}$ such that 
$\pi_P(x_*(l))=\pi_P(w_i)$
and $V$ separates $u\cup \{w_i\}$ and $v\cup x_*(l)$.
Apply the Abstract Nash-Williams Theorem to each of
the partitions of $\mathcal{K}^{w_i}$ induced by
 $\mathcal{K}_l^{w_i}$,  $l<l(w_i)$,
to obtain an $X_{i+1}\in [d_i+1,X_i]$ homogeneous for each of these partitions.
Let $d_{i+1}$ be least above $d_i$ such that $X_{i+1}(d_{i+1})\in X_*$ and let $w_{i+1}$ denote $X_{i+1}(d_{i+1})$.
Let $Z=\bigcup_{i<\om} r_{d_i +1}(X_i)$.

\begin{subclaim}\label{subclm.1}
For each $x\in\mathcal{K}|Z$ and each $l<l(x_*)$,
if $V$ mixes $u\cup x_*(0)$ and $v\cup x_*(l)$,
then $\pi_P(x_*(0))=\pi_P(x_*(l))$.
\end{subclaim}

\begin{proof}
Suppose there is an $x\in\mathcal{K}|Z$ and $l<l(x_*)$  such that $V$ mixes 
$u\cup x_*(0)$ and $v\cup x_*(l)$,
but $\pi_P(x_*(0))\ne\pi_P(x_*(l))$.
Then letting  $w_{\wedge}=x_*(0)\wedge x_*(l)$, the meet of $x_*(0)$ and $x_*(l)$, 
we see that $w_{\wedge}\tl\pi_P(x_*(l))$.
Thus, there is a $y\in\mathcal{K}|Z$ such that $y_*(0)=x_*(0)$ (hence $l(y_*)=l(x_*)$) and $\pi_P(y_*(l))\ne\pi_P(x_*(l))$.
Let $i$ be the integer such that $x_*(0)=w_i$.
Since $Z$ is homogeneous for the partition of $\mathcal{K}^{w_i}$ induced by $\mathcal{K}^{w_i}_l$,
it follows that $V$ mixes $u\cup y_*(0)$ and $v\cup y_*(l)$.
Since $x_*(0)=y_*(0)$, transitivity of mixing implies that $V$ mixes 
$v\cup x_*(l)$ and $v\cup y_*(l)$, contradicting that  $\pi_P(y_*(l))\ne\pi_P(x_*(l))$.
\end{proof}

It follows from Subclaim \ref{subclm.1}
that no $x\in\mathcal{K}|Z$ can satisfy (b).

\begin{subclaim}\label{subclm.2}
For each $x\in\mathcal{K}|Z$,
if there is an $l\le l(x_*)$ satisfying (a),
then for all $j\le l(x_*)$ with $\pi_P(x_*(j))=\pi_P(x_*(0))$, 
$V$ separates $u\cup x_*(0)$ and $v\cup x_*(j)$.
\end{subclaim}

\begin{proof}
Suppose that  $x\in\mathcal{K}|Z$ and $l\le l(x_*)$ satisfy $\pi_P(x_*(0))=\pi_P(x_*(l))$ and $V$ separates $u\cup x_*(0)$ and $v\cup x_*(l)$.
Let $j<l(x_*)$ be any index such that $\pi_P(x_*(j))=\pi_P(x_*(l))$.
Then $V$ mixes $v\cup x_*(l)$ and $v\cup x_*(j)$.
Then $V$ must separate $u\cup x_*(0)$ and $v\cup x_*(j)$
(for the alternative would imply $V$ mixes $u\cup x_*(0)$ and $v\cup x_*(l)$).
\end{proof}

Now if $\mathcal{K}|Z\sse\mathcal{K}_0$,
then it cannot be because of satisfying (b), by Subclaim \ref{subclm.1},
so it must be because (a) holds for each $x\in\mathcal{K}|Z$.
Subclaim \ref{subclm.2} then implies that $Z$ separates $u$ and $v$, a contradiction.
Therefore, $\mathcal{K}|Z\cap\mathcal{K}_0$ must be empty.
This proves the claim.
\end{proof}

By Claim \ref{claim.mixiff=} and Lemma \ref{lem.hered},
the Lemma holds.
\end{proof}

This last section of the proof of the Ramsey-classification Theorem involves putting the previous lemmas, claims, and facts together to show that 
$u\, R\, v$   if and only if $\vp'(u)\setminus\{\emptyset\}=\vp'(v)\setminus\{\emptyset\}$ if and only if  $\vp(u)=\vp(v)$, for all $u,v\in\mathcal{F}|W$, and checking that $\vp$ is in fact irreducible.

It is important to point out the following facts about the projection maps.

\begin{facts}\label{facts.superusefulstructureresults}
Let $u\in\hat{\mathcal{F}}|W$.
For $i<|u|$, let $p_i^u$ denote $\pi_{r_i(u)}(u(i))$.
\begin{enumerate}
\item
If $p_i^u\ne\emptyset$,
then $p^u_i\tr w_{r_i(u)}$.
\item
If $p^u_i\ne\emptyset$, then whenever $j<|u|$ with $j\ne i$ and $p^u_j\tre p_i^u$,
it must be the case that   $j>i$ and $p^u_j\tr p^u_i$.
\item
Let $I_u$ be the subset of $i<|u|$ for which $p^u_i\ne\emptyset$.
Then the set $\{p^u_i:i\in I_u\}$ has no duplicates; that is,
 $i\ne j$ in $I_u$ implies  $p^u_i\ne p^u_j$.
Thus, $\vp'(u)\setminus\{\emptyset\}=\{p^u_i:i\in I_u\}$.
\item
For $i,j\in I_u$,  $i<j$ implies $p^u_i\prec p^u_j$.
\end{enumerate}
\end{facts}

\begin{proof}
Suppose $p^u_i\ne\emptyset$.
Then $P_{r_i(u)}\tr\{w_{r_i(u)}\}$,
so $p^u_i\tr w_{r_i(u)}$ and hence (1) holds.

Fix $i<|u|$ and suppose $j<|u|$ is given satisfying  $i\ne j$ and $p^u_j\tre p^u_i$.
Recall that  the definition of  $w_{r_i(u)}$ implies that $w_{r_i(u)}=u(i)\wedge r_i(u)$.
Further,  there are no members  $w\in u$ such that both $w\tre p^u_i$ and $w<_{\mathrm{lex}} u(i)$, since that would contradict the definition of $w_{r_i(u)}$ and that $w_{r_i(u)}\tl p^u_i$.
So if  $j<i$ and $u(j)<_{\mathrm{lex}} u(i)$, then $p_j^u$ cannot equal $p^u_i$.
For the other case, if $j<i$ and $u(i)<_{\mathrm{lex}} u(j)$,
then $u(i)\wedge u(j)\tle w_{r_i(u)}$, which is $\tl p^u_i$.
Hence,
$u(i)\wedge u(j)\not\tre p^i_u$, so again, $p^j_u$ cannot equal $p^u_i$.

For $j>i$, if $p^u_j= p^u_i$, then $w_{r_j(u)}\tre p^u_i$, since $w_{r_j(u)}=u(j)\wedge r_j(u)$.
But this implies  $p^u_j$ must be $\emptyset$, since  projection of $u(j)$ to $w_{r_j(u)}$ or some initial segment of it implies that $\pi_{r_j(u)}$ is the projection to the emptyset.
Thus, (2) holds.
(3) follows from (2).

For (4), 
notice that for any $j\in I_u$,
$p^u_j\ne\emptyset$ implies that $p^u_j\tr w_{r_j(u)}$,
and since $\min(u(j)\setminus w_{r_j(u)})$ is strictly greater than $\max(r_j(u))$,
it follows that $p^u_j\succ $ every member of $\widehat{r_j(u)}$.
Thus, if  $i<j$, then $p^u_i\in \widehat{r_j(u)}$ implies $p^u_i\prec p^u_j$.
\end{proof}

Let $\vp''(u)$ denote $\vp'(u)\setminus\{\emptyset\}$.
Then $\vp''(u)=\{p^u_i:i\in I_u\}$.

\begin{claim}\label{claim4.19}
For $u,v\in \hat{\mathcal{F}}|W$, if $\vp''(u)=\vp''(v)$,
then $W$ mixes $u$ and $v$.
Hence, for $u,v\in\mathcal{F}|W$, if $\vp''(u)=\vp''(v)$, then $u\, R\, v$.
\end{claim}

\begin{proof}
Suppose that $\vp''(u)=\vp''(v)$.
Using the notation in Fact \ref{facts.superusefulstructureresults},
let $\lgl p^u_i:i\in I_u\rgl$ and $\lgl p^v_j:j\in I_v\rgl$ denote the $\prec$-increasing enumerations of  $\vp''(u)$ and 
 $\vp''(v)$, respectively.
Since $\vp''(u)=\vp''(v)$, $|I_u|$ must equal $|I_v|$, so let $k=|I_u|$.
Enumerate $I_u$ as $i_0<\dots<i_{k-1}$
and $I_v$ as $ j_0<\dots<j_{k-1}$.
Since  the sequences 
 $\lgl p^u_i:i\in I_u\rgl$ and $\lgl p^v_j:j\in I_v\rgl$ 
are both in $\prec$-increasing order, the only way the sets $\vp''(u)$ and $\vp''(v)$ 
 can be equal is if for each $l<k$, $p^u_{i_l}=p^v_{j_l}$, which we will denote by
 $p_l$.
Given $l<k$,
 let $x_l$ denote $r_{i_l}(u)$ and let $y_l$ denote $r_{j_l}(v)$.
Note that $\vp''(x_l)=\vp''(y_l)$.
Since $p_l\ne\emptyset$, it follows that 
$w_{x_l}\tl p_l\tle u(m_l)$ and $w_{y_l}\tl p_l\tle v(n_l)$.
Therefore, $x_l\prec p_l\preceq u(m_l)$ and $y_l\prec p_l\preceq v(n_l)$.

Now $\vp''(x_0)=\vp''(y_0)=\emptyset$, so $W$ mixes $x_0$ and $y_0$, by Fact \ref{claim.4.17} (2).
Then Lemma \ref{lem.phiknows} implies that $W$ mixes $r_{m_0+1}(u)$ and $r_{n_0+1}(v)$, since $p^u_{i_0}=p^v_{j_0}$ (and hence $u(i_0)$ and $v(j_0)$ are $\prec$-above both $x_0$ and $y_0$).
Then for all $m\in(i_0+1,i_1]$ and all $n\in(j_0+1,j_1]$, $\vp''(r_m(u))=\vp''(x_0)=\vp''(y_0)=\vp''(r_n(v))$, so Fact \ref{claim.4.17} (2)  implies that $W$ mixes $x_1$ and $y_1$.
Continuing in this manner, one proves by induction on $l<k$ that $W$ mixes $r_{m_k+1}(u)$ and $r_{n_k+1}(v)$.
Since $\vp''(u)=\vp''(r_{i_{k-1}+1}(u))=\vp''(r_{j_{k-1}+1}(v))=\vp''(v)$, apply Fact \ref{claim.4.17} (2) as necessary to conclude that $W$ mixes $u$ and $v$.
In particular, if $u,v\in\mathcal{F}|W$, then  $W$ mixing $u$ and $v$  implies that $u\, R\, v$.
\end{proof}


We point out the following  useful observation:
For $u\in\mathcal{F}|W$ and $i<|u|$, if $\pi_{r_i(u)}(u(i))\ne\emptyset$,
then $\pi_{r_i(u)}(u(i))\not\in \widehat{r_i(u)}$.
Thus, $\vp''(r_i(u))=\vp''(u)\cap\widehat{r_i(u)}$.

\begin{claim}\label{claim.phiNW}
$\vp''$ is Nash-Williams on $\mathcal{F}|W$.
\end{claim}

\begin{proof}
Let $u,v\in\mathcal{F}|W$ and suppose $m<|u|-1$ is maximal such that $\vp''(r_m(u))=\vp''(v)$ but $\vp''(r_{m+1}(u))\ne\vp''(v)$.
Then $P_{r_m(u)}\ne\{w_{r_m(u)}\}$.
Letting $k=|\vp''(u)|$ and $l=|\vp''(v)|$,
 and $\lgl p^u_i:i<k\rgl$ and $\lgl p^v_i:i<l\rgl$ being their $\prec$-increasing enumerations, 
we see that 
$k>l$,
$\lgl p^u_i:i<l\rgl=\lgl p^v_i:i<l\rgl$,
 and   $p_{l-1}^v=p_{l-1}^u\prec p_{l}^u$.

Let $n$ be maximal such that $r_n(v)\prec p^u_l$.
Then $\vp''(r_n(v))=\vp''(v)$, and if $n<|v|$, then $\pi_{r_n(v)}=\pi_{\emptyset}$.
Since $\vp''(r_m(u))=\vp''(r_n(v))$,  Claim \ref{claim4.19} implies that $W$ mixes  $r_m(u)$ and $r_n(v)$.
By Fact \ref{claim.4.17} (1), $W$ mixes $r_n(v)$ and $r_m(u)\cup \{w\}$ for at most one $\pi_{r_m(u)}$-projection class of $w$'s in $W_{r_m(u)}$. 
Since $p^u_l \succ r_n(v),r_m(u)$,  letting $d=\depth_W(r_n(v),r_m(u))$,
 there is a $W'\in [d,W]$
 such that   the least member in $W'$ above $r_n(v)$ and $r_m(u)$ is also $\prec$-above $p_l^u$.
By Lemma \ref{lem.hered}, we may assume that we already thinned $W$ to  satisfy that no 1-extension of  $r_m(u)$ into $W$ is mixed with $r_n(v)$.
But then $W$ separates $r_n(v)$ and $r_m(u)$, a contradiction.
\end{proof}

\begin{claim}\label{claim.miximpliesphi=}
For all  $u,v\in\mathcal{F}|W$, if 
$u\, R\, v$, then 
 $\vp''(u)=\vp''(v)$.
\end{claim}

\begin{proof}
Let $u,v\in\mathcal{F}|W$ and suppose that $\vp''(u)\ne\vp''(v)$.
Let $\lgl p^u_i:i< |\vp''(u)| \rgl$ and $\lgl p^v_i:i< |\vp''(v)| \rgl$ enumerate $\vp''(u)$ and $\vp''(v)$ in $\prec$-increasing order.
Since  by Claim \ref{claim.phiNW} $\vp''$ is Nash-Williams,
neither of 
$\vp''(u)$, $\vp''(v)$ end-extends the other;
so let  $j<\min(|\vp''(u)|,|\vp''(v)|)$ be least such that $p^u_j\ne p^v_j$.
Let $m_j,n_j$ be the integers such that $\pi_{r_{m_j}}(u(m_j))=p^u_j$ and $\pi_{r_{n_j}}(v(n_j))=p^v_j$.

If both $u(m_j)$ and $v(n_j)$ are $\prec$-above both $r_{m_j}(u)$ and $r_{n_j}(v)$, then Lemma \ref{lem.phiknows}  along with $p^u_j\ne p^v_j$ imply that $r_{m_j+1}(u)$ and $r_{n_j+1}(v)$  are separated by $W$; and hence, $u\not\hskip-.06in R\, v$.
Otherwise, without loss of generality, suppose $u(m_j)\preceq r_{n_j}(v)$.
Take $n$ least such that $v(n)\succeq u(m_j)$.
Then $n< n_j$, $r_n(v)\prec u(m_j)$, and $v(n)\succ r_{m_j}(u)$.
Assume that $W$ mixes $r_n(v)$ and $r_{m_j}(u)$, for if not, then already 
we have $u\not\hskip-.06in R\, v$.
Since $n< n_j$, $\pi_{r_n(v)}(v(n))$ cannot equal $p^u_j$.
By Lemma \ref{lem.phiknows}, $W$ separates $r_n(v)$ and $r_{m_j}(u)$; hence  $u\not\hskip-.06in R\, v$.
\end{proof}

Claims  \ref{claim4.19}  and \ref{claim.miximpliesphi=} 
yield  that  $\vp''$ canonizes the equivalence relation $R$ on $\mathcal{F}|W$; that is, for
 each pair $u,v\in\mathcal{F}|W$,
$u\, R\, v$ if and only if $\vp''(u)=\vp''(v)$.

Next, we show that $\vp$ holds the same information about $R$  that $\vp''$ does. 
The following observation will be useful:
For  each $u\in\hat{\mathcal{F}}|W$,
\begin{equation}\label{eq.obs}
\pi_{r_0(u)}(u(0))\preceq u(0)\prec \pi_{r_1(u)}(u(1))\preceq u(1)\prec\cdots.
\end{equation}
Recall that for any two sets $D,E$, $D\triangle  E$ denotes their symmetric difference, $(D\setminus E)\cup(E\setminus D)$.

\begin{lem}\label{lem.vp''=vp}
For $u,v\in \hat{\mathcal{F}}|W$, $\vp''(u)=\vp''(v)$ if and only if $\vp(u)=\vp(v)$.
Therefore, for $u,v\in \mathcal{F}|W$, $u\, R\, v$ if and only if $\vp(u)=\vp(v)$.
\end{lem}

\begin{proof}
Since $\vp(u)$ is the set of all maximal nodes in  $\vp''(u)$, the forward direction is trivial.
So now assume that $\vp''(u)\ne\vp''(v)$.
Let $w$ denote the $\prec$-least member of $\vp''(u)\triangle \vp''(v)$, and without loss of generality, assume that $w\in \vp''(u)\setminus \vp''(v)$.
Let $i<|u|$ and $j<|v|$ be maximal such that $r_i(u)\prec w$ and $r_j(v)\prec w$.
Then $w\preceq u(i)$ and $w\preceq v(j)$.

Since  $w$ is $\prec$-minimal in $\vp''(u)\triangle \vp''(v)$, $w\in\vp''(u)$, 
and $r_i(u)$ is $\prec$-maximal below $w$, it follows from the observation 
(\ref{eq.obs}) that $w=\pi_{r_i(u)}(u(i))$.
Hence, 
the set $\{\pi_{r_l(u)}(u(l))\prec w:l<|u|\}$ equals $\vp''(r_i(u))$.
Let $w'$ denote $\pi_{r_j(v)}(v(j))$.
Then $w'$ is the $\prec$-least member of $\vp''(v)$ satisfying $w\prec w'$, so by observation (\ref{eq.obs}), we see that
 $\{\pi_{r_l(v)}(v(l))\prec w:l<|v|\}=\vp''(r_j(v))$.
Therefore,  $\vp''(r_i(u))=\vp''(r_j(v))$, since $w$ is  the $\prec$-least member of $\vp''(u)\triangle \vp''(v)$.

Note that $r_j(v)\prec w=\pi_{r_i(u)}(u(i))\preceq u(i)$ and $r_i(u)\prec w\prec w'\preceq v(j)$,
so $u(i)\succ r_j(v)$ and $v(j)\succ r_i(u)$.
Since $\vp''(r_i(u))=\vp''(r_j(v))$, Claim \ref{claim4.19} implies that $W$ mixes $r_i(u)$ and $r_j(v)$.  Hence, $P_{r_i(u)}$ and $P_{r_j(v)}$ are equal on $W$ above $r_i(u)$ and $r_j(v)$, by the proof of Lemma \ref{lem.phiknows}.
Since $P_{r_i(u)}$ is a uniform projection, $w\ne w'$ implies that $w$ and $w'$ are $\tle$-incomparable.
Hence, $w\not\tle v(j)$.
This along with the fact that $r_j(v)\prec w$ (and hence $w\not\in\widehat{r_j(v)}$) imply that
 $w\not\in\widehat{r_{j+1}(v)}$.
For $j< l<|v|$,  $\min(v(l)\setminus w_{r_l(v)})=\min(v(l)\setminus r_l(v))\ge \max(w')>\max(w)$.
Thus,
 $\pi_{r_j(v)}(v(j))\not\tre w$, since $w\not\in\widehat{r_{j+1}(v)}$ and every new number in $v(j)\setminus \widehat{r_{j+1}(v)}$ is strictly greater than $\max(w)$.
Hence, $w\not\in\widehat{\vp''(v)}$, which implies that
$\vp(u)\setminus \vp(v)\ne\emptyset$.
Thus, we have proved that whenever $\vp''(u)\ne\vp''(v)$, then also $\vp(u)\ne\vp(v)$, which proves the first statement of the lemma.
By Claims 
\ref{claim4.19} and
\ref{claim.miximpliesphi=}, $\vp$ canonizes $R$ on $\mathcal{F}|W$.
\end{proof}

\begin{fact}\label{fact.vpNW}
 $\vp$ is  Nash-Williams on $\mathcal{F}|W$.
\end{fact}

\begin{proof}
If $u,v\in\mathcal{F}|W$ and there is some $j<|v|$ such that $\vp(u)=\vp(v)\cap\widehat{r_j(v)}$,
then noting that this implies that $\vp(u)=\vp(r_j(v))$, 
Lemma \ref{lem.vp''=vp} implies that
$\vp''(u)=\vp''(r_j(v))$.
Since $\vp''$ is Nash-Williams, by Claim  \ref{claim.phiNW},
$\vp''(u)$ and $\vp''(v)$ must be equal.
Then  Lemma \ref{lem.vp''=vp} implies that $\vp(u)=\vp(v)$.
\end{proof}

The next lemma will conclude the proof of the Ramsey-classification Theorem.
Essentially, the property $(*)$ follows from taking $W'$ thin enough below $W$ and then applying Claim \ref{claim4.19}  and Lemma \ref{lem.phiknows}.

\begin{lem}\label{vpinner}
 $\vp$ is irreducible.
\end{lem}

\begin{proof}
Since  $\vp$ is inner, it suffices to show that $\vp$ satisfies property $(*)$.
Take $W'\le W$ thin enough so that the following holds:
There are $n(0,0)<\dots<n(0,l_0)$, where $l_0=\lh(W'(0))-1$, such that for each $i\le l_0$, $W(n(0,i))\prec W'(0)$ and $\lh (W(n(0,i))\wedge W'(0))=i$.
In general, we require $W'$ thin enough that for each $k\ge 1$,
there are $n(k,m_k)<\dots<n(k,l_k)$, where $m_k=\lh(w_{r_k(W)})$ and  $l_k=\lh(W'(k))-1$, such that for each $m_k\le i\le l_k$, $W'(k-1)\prec W(n(k,i))\prec W'(k)$ and $\lh (W(n(k,i))\wedge W'(k))=i$.

Let $u\in\mathcal{F}|W'$ and build $v\in\mathcal{F}|W$ as follows.
If $\pi_{r_0(u)}(u(0))=u(0)$, then let $v(0)=u(0)$.
Otherwise,
take $v(0)\prec u(0)$  in $W$ such that $v(0)\wedge u(0)=\pi_{r_0(u)}(u(0))$.  
Since $P_{r_0(u)}$ is a uniform projection and $v(0)\tre \pi_{r_0(u)}(u(0))$, it follows that 
$\pi_{r_0(u)}(v(0))=\pi_{r_0(u)}(u(0))$.
Thus, in both cases, letting $r_1(v)$ denote $\{v(0)\}$, we see that $\vp(r_1(v))=\vp(r_1(u))$ and  $r_1(v) \wedge u=\vp(r_1(u))$.

Suppose that $n\le |u|$ and we have chosen $r_n(v)=\{v(0),\dots,v(n-1)\}$ so that 
for each $i<n$, $\pi_{r_i}(v(i))=\pi_{r_i}(u(i))$ and $v(i)\wedge u=\pi_{r_i}(u(i))$.
Thus, $\vp(r_n(v))=\vp(r_n(u))= r_n(v)\wedge u$.
There are three possibilities.

If $r_n(v)\in\mathcal{F}$, then let $v=r_n(v)$.
By the induction hypothesis,
$\vp(v)=\vp(r_n(u))=v\wedge u$.  Since $\vp$ is Nash-Williams,  $\vp(u)$ must equal $\vp(v)$.

If  $n<|u|$ and $r_n(v)\not\in\mathcal{F}$,
recall that since $\vp(r_n(v))=\vp(r_n(u))$,
 the proof of Lemma \ref{lem.phiknows} shows that $\pi_{r_n(v)}=\pi_{r_n(u)}$.
Thus, choose $v(n)\tr w_{r_n(v)}$ so that
$u(n-1)\prec v(n)\prec u(n)$,
 $\pi_{r_n(v)}(v(n))=\pi_{r_n(u)}(u(n))$,
and $v(n)\wedge u=\pi_{r_n(v)}(v(n))$.
(This last condition is possible by taking $v(n)$ so that $\min (v(n)\setminus \pi_{r_n(v)}(v(n)))$ is not in $u$.)
Then   $\vp(r_{n+1}(v))=\vp(r_{n+1}(u))=r_{n+1}(v)\wedge u$.

If  $n=|u|$ and  $r_n(v)\not \in\mathcal{F}$,
let $w_*$ denote $w_{r_n(v)}\wedge u$.
Note that $w_*$ is in $\widehat{r_n(v)}\wedge \hat{u}$, since $w_{r_n(v)}\in \widehat{r_n(v)}$.
Thus, 
by  choosing any $v(n)\tr w_{r_n(v)}$ with $\min(v(n)\tr w_{r_n(v)})>\max(u\cup r_n(v))$, we obtain an $r_{n+1}(v)=r_n(v)\cup\{v(n)\}$
satisfying $v(n)\wedge u=w_*$.
Thus, $r_{n+1}(v)\wedge u=r_n(v)\wedge u$, which by the induction hypothesis is equal to $\vp(u)$ and $\vp(r_n(v))$.  
Since $\vp$ is Nash-Williams, any 
 extension of $r_n(v)$ to some $v\in\mathcal{F}$ will have $\vp(v)=\vp(r_n(v)))$.
Hence, also $\vp(r_{n+1}(v))=\vp(r_n(v))=\vp(u)$.
\end{proof}

Thus, $\vp$ is an irreducible map canonizing $R$ on $\mathcal{F}|W$.
Thus concludes the proof of Theorem \ref{thm.rc}.
\end{proof}


This section concludes with proving that the irreducible map in Theorem \ref{thm.rc} is unique.  First, some useful observations.

\begin{fact}
$u(i)\wedge \vp(u)\tr w_{r_i(u)}$ if and only if $\pi_{r_i(u)}(u(i))\ne\emptyset$.
\end{fact}

\begin{fact}\label{fact.trmaxl}
$\vp$ is $\tle$-maximal among projection maps canonizing $R$ on $\mathcal{F}|W$.
That is, if $\gamma$ is another projection map such that for all $u,v\in\mathcal{F}|W$, $\gamma(u)=\gamma(v)\llra u\, R\, v$,
then $\gamma(u)\tle\vp(u)$, for each $u\in\mathcal{F}|W$.
\end{fact}

\begin{proof}
Let $u\in\mathcal{F}|W$, and suppose $n<|u|$ 
is such that $\gamma(u)\cap\widehat{u(n)}\ne\vp(u)\cap\widehat{u(n)}$.
Let $w$ denote $u(n)$.
Since $\gamma$ is a projection map, then either 
$\gamma(u)\cap\hat{w}\tl \vp(u)\cap \hat{w}$  or
$\gamma(u)\cap\hat{w}\tr \vp(u)\cap \hat{w}$.
Suppose toward a contradiction that $\gamma(u)\cap\hat{w}\tr \vp(u)\cap \hat{w}$.
Let $x$ denote $r_n(u)$.
Take $w'\in W$ such that $x\cup\{w'\}\in r_{n+1}[x,W]$
and $\pi_x(w')=\pi_x(w)$.
Then $x\cup\{w\}$ and $x\cup\{w'\}$ are mixed, so we may extend $x\cup \{w\}$ and $x\cup\{w'\}$ to some $y,y'\in\mathcal{F}|W$, respectively, such that $\vp(y)=\vp(y')$.
Then $y\, R\, y'$.
Now $\gamma(y)\cap \hat{w}\ne\gamma(y')\cap \hat{w}'$; so since $\gamma$ is a projection map, it follows that $\gamma(y)\ne\gamma(y')$, which contradicts that $\gamma$ canonizes $R$ on $\mathcal{F}|W$.
\end{proof}

We now prove that irreducible maps are unique, up to restriction below some member of the space.
\vskip.1in

\noindent \it Proof of Theorem \ref{thm.irred}. \rm
Let $\gamma$ be any irreducible map canonizing $R$ on $\mathcal{F}|W$.
Take $W'\le W$ witnessing  $(*)$ for $\gamma$.
Suppose toward a contradiction that there is a   $u\in\mathcal{F}|W'$
such that $\gamma(u)\ne \vp(u)$.
Then $\gamma(u)\tl \vp(u)$, by Fact \ref{fact.trmaxl}.
Take  a $v\in \mathcal{F}|W$ such that   $\gamma(u)=\gamma(v)=u\wedge v$.
Then $u\wedge v\tl \vp(u)$.
Since $\vp(u)$ and $\vp(v)$ are projections of $u$ and $v$, respectively, if $\vp(u)$ and $\vp(v)$ are to be equal, then they must both be contained in $u\wedge v$, which is not the case.
Thus, $\vp(u)\ne\vp(v)$, a contradiction.
\hfill$\square$


\section{Conclusion and Remarks}

We close this paper with a few brief remarks.
The canonical equivalence relations on fronts of the form $\mathcal{AE}^B_k$ are obtained as a corollary to Theorem \ref{thm.rc}.  We do not present it here, since, unlike the analogous situation for all other known topological Ramsey spaces, here we do not obtain simply a sequence of $k$ many canonical equivalence relations on the 1-extensions.  This happens because, fixing $i<k$, for $X,Y\in\mathcal{E}_B$,  $r_{i+1}[i,X]$ and $r_{i+1}[i,Y]$ may be isomorphic to $\mathcal{E}_C$ and $\mathcal{E}_D$ for very different uniform barriers $C$ and $D$.

We have shown in Theorem \ref{thm.irred} that the irreducible map $\vp$ in Theorem \ref{thm.rc} is unique; more to the point, the property $(*)$ is  sufficient for uniqueness among inner maps canonizing the equivalence relation $R$.
We ask whether the following property is sufficient for uniqueness:  We say that a projection map $\gamma$ on a front $\mathcal{F}$ on $\mathcal{E}_B$ is {\em strongly Nash-Williams} if for each pair $u,v\in\mathcal{F}$ and each $i\le \depth_W(u,v)$, 
$\gamma(u)\cap\widehat{r_i(W)}\tle\gamma(v)\cap\widehat{r_i(W)}$ 
 implies $\gamma(u)\cap\widehat{r_i(W)}=\gamma(v)\cap\widehat{r_i(W)}$.
If a projection map $\gamma$ canonizing $R$ on $\mathcal{E}_B|W$ is strongly Nash-Williams, is it necessarily equal to $\vp$?
We leave this question for future work.

We conclude by mentioning that the topological Ramsey spaces $\mathcal{E}_B$ and the Ramsey-classification Theorem \ref{thm.rc} form the basis for forthcoming work \cite{DobrinenE_BTukey15} on the Rudin-Keisler and Tukey structures of the associated ultrafilters $\mathcal{G}_B$.


\bibliographystyle{ieeetr}

\bibliography{references}

\end{document}